\documentclass{amsart}
\usepackage{amsmath}
\usepackage{array}
\usepackage{amsfonts}
\usepackage{amsaddr}
\usepackage{url} 
\usepackage{fancyhdr}
\usepackage{multirow}
\usepackage{graphicx}
\usepackage{relsize}
\usepackage{amssymb}  
\usepackage{hyperref}
\usepackage{mathtools}
\usepackage{float}
\usepackage{indentfirst}
\usepackage{enumerate}

\usepackage{caption}
\usepackage{subcaption}

\newtheorem{theorem}{Theorem}[section]
\newtheorem{lemma}[theorem]{Lemma}
\newtheorem{corollary}[theorem]{Corollary}

\newtheorem{definition}[theorem]{Definition}

\newtheorem{remark}[theorem]{Remark}

\hypersetup{colorlinks=true,linkcolor=blue, linktocpage}
\hypersetup{colorlinks,
citecolor=blue,
allcolors=blue
}
\DeclareMathOperator{\Int}{int}
\DeclareMathOperator{\dist}{d}

\DeclareMathOperator{\diam}{diam}
\DeclareMathOperator{\htop}{h_{top}}
\DeclareMathOperator{\Rec}{Rec}
\DeclareMathOperator{\End}{End}
\DeclareMathOperator{\Per}{Per}

\title[Entropy in hyperspace]{On recurrence and entropy in hyperspace of continua in dimension one}
\author[D. Jeli\'c]{Domagoj Jeli\'c}
\address[D. Jeli\'c]
{Faculty of Science, University of Split, Ru\dj era Bo\v{s}kovi\'ca 33, 21000 Split, Croatia
}
\author[P. Oprocha]{Piotr Oprocha}
\address[P. Oprocha]{AGH University of Science and Technology, Faculty of Applied
	Mathematics, al.
	Mickiewicza 30, 30-059 Krak\'ow, Poland
	-- and --
	Centre of Excellence IT4Innovations - Institute for Research and Applications of Fuzzy Modeling, University of Ostrava, 30. dubna 22, 701 03 Ostrava 1, Czech Republic.}
\AtEndDocument{\emph{E-mail addresses:} djelic@pmfst.hr, oprocha@agh.edu.pl}
\date{\today}
\begin{document}
\begin{abstract}
	We show that if $G$ is a topological graph, and $f$ is continuous map, then the induced map $\tilde{f}$	
	acting on the hyperspace $C(G)$ of all connected subsets of $G$ by natural formula $\tilde{f}(C)=f(C)$ carries the same entropy as $f$.
	This is well known that it does not hold on the larger hyperspace of all compact subsets. Also negative examples were given for the hyperspace $C(X)$ on some continua $X$, including dendrites.
	
	Our work extends previous positive results obtained first for much simpler case of compact interval by completely different tools.
\end{abstract}
\maketitle
\section{Introduction}

	Whenever we are given a selfmap $f$ of a compact metric space $X$,
we can associate with it induced mappings $\overline{f}$ and $\tilde{f}$ on the hyperspace $2^X$ of compact subsets of $X$ and the hyperspace $C(X)$ of continua in $X$, respectively, both defined in a natural way.
	If $2^X$ is endowed with the so-called \emph{Hausdorff metric} (see below) then $\overline{f}$ and its restriction $\tilde{f}$ are continuous. Both of them also have richer dynamics than $f$, which can be viewed as their subsystem.
 	Then a natural question which arises immediately,  is to what extent the dynamical properties of the base map $f$ and these maps are similar. 
 	The first study of this relation between \emph{individual} and \emph{collective} dynamics dates back to 1975 and was done by Bauer and Sigmund in~\cite{Bauer}.
 	Since then, a lot of progress regarding this topic has been made by many authors (e.g. see~\cite{KOCSF, Banks, Lampart}).
  	In the present paper, we will be interested in the relation between topological entropies of the base map $f$ and of its induced map on continua $\tilde{f}$.
	The main aim of this paper is proving the following theorem.

\begin{theorem}\label{tm:entropy}
	Let $(G,f)$ be dynamical system on a topological graph and let $(C(G),\tilde{f})$ be the dynamical system induced by $f$ on the hyperspace of all connected sets. Then $\htop(f)=\htop(\tilde{f}).$
\end{theorem}

	The analogous results to the one stated have been obtained for simpler spaces, first for compact intervals and later for topological trees by Matviichuk (\cite{MatviiInterval, Matviichuk}).
	A simplified version of Theorem~\ref{tm:entropy} was proved for the highly restricted case of transitive graph maps in~\cite{KOCSF} (cf. \cite{KOCSF}) but the proof for general graph maps has not been provided so far.
	It is important to note that this property holds essentially due to dimension one, 
as for higher dimensional spaces there are simple examples of base maps having zero entropy whose induced maps on continua have infinite entropy. 
	Even for some one-dimensional spaces like dendrites, such examples have been constructed (see~\cite{Dendrite1} or~\cite{Dendrite2}).
	It is also well known that the analogous statement for the induced maps of hyperspace of all compact subsets does not hold, since for each map $f$ having positive topological entropy, the entropy of $\overline{f}$ is infinite~\cite{KOCSF}. On the other hand, if we consider the space of all probability invariant measures in weak* topology, then the entropy does not increase \cite{Bauer}.

		Usually, when one passes from interval or tree ambient to topological graphs, they encounter various difficulties which emerge as a consequence of existence of circles.
	The first is the fact that if two different points of a graph are given, generally there is no unique minimal arc connecting them, unlike for trees.
	The second is the phenomenon of irrational rotation and its extensions which act by far differently than any tree map.
	
	
	To prove the main theorem, we first investigate the structure of recurrent points of $\tilde f$. As a first important step, we provide full characterization of the set $\Rec(\tilde f)$
	in Theorem~\ref{thm:rec-cont}. We obtain that the most of nondegenerate recurrent continua are periodic, except possibly some of them related to irrational rotations. This allows us to localize maps with positive entropy, which by variational principle reveals the value of topological entropy.
	
 The paper is organized as follows: In Section~\ref{sec:preliminaries} we recall basic definitions and some properties of graph maps and hyperspace dynamics.
	Section~\ref{sec:recurrent-vs-limitsets} is devoted to some auxiliary results used in Section~\ref{sec:characterization} where we provide characterization of the set of recurrent continua.
	Finally, in Section~\ref{sec:proof} we prove the main result of this paper.

\section{Preliminaries}\label{sec:preliminaries}

\subsection{Topological graphs}
	A \emph{topological graph} (or short \emph{graph}) is a 
	continuum $G$ 
	 such that there is a one dimensional simplical complex $K$ whose geometric carrier $|K|$ is homeomorphic to $G$.
	 If $G$ is a singleton then we call it a degenerate graph.
	
	If $G$ is a graph, then there is a finite set $V\subset G$ such that each connected component of $G\setminus V$ is homeomorphic to an open interval.
	We can assume that $V$ is minimal subset of $G$ with this property, i.e. that for each $v\in V$ and each connected neighborhood $U$ of $v$ in $G$, such that $U\cap V=\{v\},$
	$U\setminus \{v\}$ has either one or at least three components.  
	Having said that, each element of $V$ will be called \emph{vertex} and
	closure in $G$ of each component of $G\setminus V$ will be called an \emph{edge} of $G$.
	Notice that each edge is either homeomorphic to a closed interval or to the circle $\mathbb{S}^1$.
	Edges homeomorphic to $\mathbb{S}^1$ will be called \emph{loops} of $G$.
	A \emph{branching point} of $G$ is a point in $G$ having no neighborhood homeomorphic to an interval.
	It is clear that each branching point is an element of $V$.
	Elements of $V$ different from branching points will be called \emph{endpoints} of $G$.
	We denote the set of all endpoints of $G$ by $\End(G)$.
	Moreover, every nondegenerate subcontinuum of $G$ will be called a \emph{subgraph}.

\begin{remark}\label{rem:ambient}
	From this moment on we identify $G$ with a geometric carrier of simplicial complex defining it.
	Strictly speaking, we assume that graph $G$ is a subset of $\mathbb{R}^3,$ its edges lie on straight lines
	and on each edge we have a metric assigning to it the length one.
		Moreover, the metric $d$ on $G$ now can be chosen in a way that $d(x,y)$ equals the length of the shortest path between $x$ and $y.$
		Induced topology matches the original subspace topology of $G$.
	\end{remark}

\subsection{Dynamical systems}
	A (discrete, topological) \emph{dynamical system} is a pair $(X,f)$ where $X$ is a compact metric space and $f\colon X\to X$ is a continuous map.
	We will often identify dynamical system $(X,f)$ with the map $f$.
	Any continuous map $f\colon G\to G$, where $G$ is a graph, will simply be called a \emph{graph map}.
	For a point $x\in X$, we define its \emph{orbit} as the set $\mathcal{O}(x)=\{f^n(x)\colon n\geq 0\}$ and its \emph{trajectory} as the sequence $\left(f^n(x)\right)_{n\geq 0}$.
	Analogously, the \emph{orbit of a set} $A\subset X$ is the set $\mathcal{O}(A)=\bigcup_{n=0}^{\infty}f^n(A).$
	A point $x\in X$ is said to be \emph{periodic} if there is some $p>0$ such that $f^p(x)=x.$
	The smallest such $p$ is called the \emph{period} of $x$.
	We denote the set of all periodic points of $f$ by $\Per(f)$.
	A point $x\in X$ is said to be \emph{eventually periodic} if there is some $n\geq 0$ such that $f^n(x)\in\Per(f).$ 
	Equivalently, a point is eventually periodic if and only if it has finite orbit.

	If $(X,f)$ and $(Y,g)$ are two dynamical systems, a \emph{semi-conjugacy} between $f$ and $g$ is any surjective continuous map $\phi\colon X\to Y$ such that $\phi\circ f=g\circ \phi$.
	If in addition $\phi$ is bijective, and hence a homeomorphism, then it is a \emph{conjugacy} between $f$ and $g$ and we say that these two dynamical systems are \emph{conjugate}.

	A set $A\subset X$ is said to be {\emph{$f$-invariant}} if $f(A)\subset A$ and {\emph{strongly $f$-invariant}} if $f(A)=A.$
	A closed invariant set $M\subset X$ without closed and invariant proper subset is called \emph{minimal}.
	It is well known that closed invariant set $M\subset X$ is minimal if and only if the orbit of each point $x\in M$ is dense in $M$.	
	If $X$ is minimal, we say that dynamical system $(X,f)$ is minimal or, simply,  
	that $f$ is \emph{minimal map}.
	If $f^n$ is minimal for all $n>0$, $f$ is said to be \emph{totally minimal}.
	
	Map $f\colon X\to X$ is \emph{transitive} if for every pair of nonempty open subsets $U, V\subset X$ there is some $n>0$ such that $f^n(U)\cap V\neq\emptyset$.
	If $X$ does not contain isolated points, then $f$ is transitive if and only if there is $x\in X$ whose orbit is dense in $X$.	
	A map $f\colon X\to X$ is \emph{mixing} if for every pair of nonempty open subsets $U, V\subset X$ there is an $N\geq 0$ such that $f^n(U)\cap V\neq \emptyset$ for all $n\geq N$.	
	It is clear that each mixing map is transitive but not vice-versa.
	For a point $x\in X$, the limit set of its trajectoy is called \emph{$\omega$-limit set of $x$} and denoted by $\omega_f(x)$.
	In other words, $y\in\omega_f(x)$ if and only if there is a strictly increasing sequence of positive integers $(n_i)_{i\geq 0}$ such that $\left(f^{n_i}(x)\right)\to y$ as $i\to\infty$.
	A point $x\in X$ is \emph{non-wandering} if for every neighborhood $U$ of $x$ and every $N>0$ there is some $n>N$ such that $f^n(U)\cap U\neq\emptyset.$	Otherwise, we call it a \emph{wandering} point.
	A point $x\in X$ is \emph{recurrent} if $x\in\omega_f(x)$. 
	The set of all recurrent points is denoted by $\Rec (f)$ and we denote $\omega(f)=\bigcup_{x\in X}\omega_f(x).$
	
		\begin{remark}\label{rem:evperiodic}
If $x\in X$ is recurrent and eventually periodic, then $x$ is periodic. 
		\end{remark}

	Each dynamical system $(X,f)$ may be assigned a \textit{topological entropy} $\htop(f)\in [0,\infty]$ which to some extent reflects its complexity.
While it is hard to tell whether higher entropy means higher topological complexity of the system (cf. \cite{HK,HKO}), it is commonly accepted fact that
complex systems start with positive entropy, while entropy zero represents tame dynamics. The reader not familiar with definition of topological entropy is
referred to standard textbooks, e.g. \cite{Walters}.

\begin{definition}
	Let $f\colon X\to X$ and $g\colon Y\to Y$ be two continuous maps of compact metric spaces $X,Y$ and $K\subset X$ be a closed invariant set.
	A semi-conjugation $\phi\colon X\to Y$ is an almost conjugacy between $f\vert_K$ and $g$ if
	\begin{enumerate}[(1)]
		\item\label{almost1} $\phi(K)=Y$,
		\item $\forall y\in Y$, $\phi^{-1}(y)$ is connected,
		\item $\forall y\in Y$, $\phi^{-1}(y)\cap K=\partial\phi^{-1}(y)$, where $\partial A$ denotes the boundary of $A$ in $X$,
		\item\label{alomst4} $\exists N\geq 1$ such that $\forall y\in Y$, $\phi^{-1}(y)\cap K$ has at most $N$ elements (and, by \eqref{almost1}, at least one element).
	\end{enumerate}
\end{definition}
When $X$ and $Y$ are graphs then the last condition is a consequence of the previous ones since both singletons and subgraphs of $G$ have boundaries of uniformly bounded cardinality.

By results of Bowen \cite{Bow71} condition \eqref{alomst4} implies that  $f|_K$ and $g$ have the same topological entropies. 

\subsection{Hyperspaces}

	A \emph{hyperspace} of a metric space $(X,\dist)$ is a specified family of nonempty closed subsets of $X$. 
	The hyperspaces which will be of our interest are $2^X$, the hyperspace of all nonempty compact subsets of $X$ and $C(X)$ which is the hyperspace of all continua, i.e. the connected elements of $2^X$.
	For a point $x\in X$ and a nonempty subset $A\subset X$ we define the \emph{distance from the point $x$ to the set $A$} as $\dist(x,A)=\inf\{\dist(x,y)\colon y\in A\}$
	and for each nonempty set $A\subset X$ and each $\epsilon>0$, we define the \emph{$\epsilon$-neighborhood of the set $A$} as $N(A,\epsilon)=\{x\in X\colon \dist(x,A)<\epsilon\}$.
	We can now endow $2^X$ with a function $\dist_H\colon 2^X\times 2^X\to\mathbb{R}$, defined for each pair $A,B\in 2^X$ as:
	\[\dist_H(A,B)=\inf\{\epsilon\geq 0\colon A\subset N(A,\epsilon)\text{ and }B\subset N(A,\epsilon)\}.\] 
	It is well known that $\dist_H$ is a metric, called \emph{Hausdorff metric}, and that $(2^X,\dist_H)$ is a compact metric space (e.g. see \cite{Nadler}). 
	It is also well know that $C(X)$ is a compact subset of $2^X$. 
	The topology generated by $\dist_H$ does not depend on the metric $\dist$ but only on the topology that $\dist$ generates on $X$.

	If $f\colon X\to X$ is a continuous map, then it naturally extends to $\bar{f}\colon 2^X\to 2^X$ via $\bar{f}(A)=f(A)$.
	It is well known that $\bar{f}$ defined this way is a continuous map on $2^X.$
	Obviously, $C(X)$ is $\bar{f}$-invariant subset of $2^X$ and we may denote $\tilde{f}=\bar{f}\vert_{C(X)}$.
	This way we obtain two \emph{induced systems} $(2^X,\bar{f})$ and $(C(X),\tilde{f})$ generated by $(X,f)$.
	Observe that $(X,f)$ may be considered a subsystem (that is, a closed invariant subset) of $(C(X),\tilde{f})$ by identifying $x\in X$ with $\{x\}\in C(X)$,
	and thus also a subsystem of $(2^X,\bar f)$. 

\subsection{Limit sets in graph maps}
	In~\cite{MaiClosed}, authors proved that the family of all $\omega$-limit sets of a graph map is closed with respect to the Hausdorff metric in the hyperspace of all closed subsets of the graph.
	As a direct consequence, $\omega$-limit set of any point in $G$ lies in a maximal $\omega$-limit set.
	

		The set $\mathcal{O}(M)=\bigcup_{i=0}^kf^i(M)$ is called a \emph{cycle of graphs of period $k$} if $M$ is a subgraph of $G$ such that $M,f(M),...,f^{k-1}(M)$ are pairwise disjoint and $f^k(M)=M.$  

 A \emph{generating sequence} or a \emph{sequence generating a solenoidal set} is any nested sequence of cycles of graphs $M_1\supset M_2\supset\cdots$ for $f$ with periods tending to infinity. By definition $Q=\bigcap_n M_n$ is closed and strongly invariant, i.e. $f(Q)=Q$. Any closed and strongly invariant subset of $Q$ is called a \emph{solenoidal set}. The characterization of Blokh (e.g. see \cite[Theorem 1]{B1}) shows that $Q$ contains a perfect minimal set $Q_{min}=Q\cap\overline{Per f}$ such that $Q_{min}=\omega(x)$, for all $x\in Q$, and a maximal $\omega$-limit set (with respect to inclusion) $Q_{max}$ such that $Q_{max}=Q\cap \omega(f)$ which we will refer to as \emph{maximal solenoid}. 

	If $x$ is a point of a graph $G$, then by a \emph{side $T$ of the point $x$} we mean a family of open, non-degenerate arcs $\{V_T (x)\}$ containing no branching points, with one endpoint at $x$, such that $\bigcap_{V\in T} \overline{V}=\{x\}$ and if $U,V\in T$, then either $V\subset U$ or $U\subset V$. Members of the family $T$ are called $T$-sided neighborhoods of $x$.\\
 Let $f\colon G\to G$ be a graph map and $M\subset G$ be a cycle of graphs. For every $x\in M$, we define the \emph{prolongation set of $x$ with respect to $f|_M$}: 
 \[P_M(x,f)=\bigcap_U\overline{\bigcup_{i=1}^{\infty}f^i(U)},\]
 where $U$ is a relative neighborhood of $x$ in $M$. If the acting map $f$ is clear from the context, we will simply write $P_M(x)$ and when $M=G$ then we will write $P(x,f)$ or simply $P(x)$.  Observe that just as $x$ being recurrent is equivalent to $x\in\omega_f(x)$, $x$ being non-wandering is equivalent to $x\in P(x)$. Obviously, $P(x)$ is an invariant closed set and the map $f|_{P(x)}$ is surjective whenever $x$ is a non-wandering point. Similarly, we define the \emph{prolongation set of $x$ with respect to a side $T$}:
 \[P^T_M(x,f)=\bigcap_{V_T(x)}\overline{\bigcup_{i=1}^{\infty}f^i(V_T(x))},\]
 where the intersection is taken over relative $T$-sided neighborhoods $V_T(x)$ of $x$ in $M$. 
 
 An arc $V\subseteq G$  is \emph{non-wandering} if there is an integer $m\geq 1$ such that $f^m(V)\cap V\neq\emptyset$. 
 If every set $V_T(x)\in T$ is non-wandering then $P^T(x)$ is one of the following (see \cite{Forys}){}:
 	\begin{itemize}
 		\item $P^T(x)$ is a periodic orbit,
 		\item $P^T(x)$ is a cycle of graphs,
 		\item $P^T(x)$ is a solenoidal set $Q$.
 	\end{itemize}
 	
Let $M\subset G$ be a cycle of graphs. We define the following sets:
\[E(M,f)=\{x\in M: P_M(x,f)=M\}\]
and
\[E_S(M,f)=\{x\in M: \text{there is a side $T$ such that } P^T_M(x,f)=M\}.\] 
Clearly, $E_S(M,f)\subseteq E(M,f)$. These sets are closed and invariant. If $E(M,f)$ is infinite then, by~\cite[Theorem 2]{B1}, $E_S(M,f)= E(M,f)$. In general, $E_S(M,f)\neq E(M,f)$ and $f(E(X, f))\neq E(X, f)$.
\begin{theorem}[Blokh, \cite{B1,B2}]\label{E}
	Let $M\subset G$ be a cycle of graphs such that $E_S(M,f)$ is non-empty. If $E_S(M,f)$ is finite then it is a periodic orbit. Otherwise, $E_S(M,f)= E(M,f)$ and it is an infinite maximal $\omega$-limit set.
\end{theorem}

Let $E(M,f)$ be an infinite maximal $\omega$-limit set from Theorem ~\ref{E}. We say that $E(M,f)$ is a \emph{basic set} if $\Per (f)\cap M\neq \emptyset$ and we denote it by $D(M,f)$, and in the opposite case when $\Per (f)\cap M= \emptyset$ we say that $E(M,f)$ is a \emph{circumferential set} and we denote it by $S(M,f)$. We will simply write $D(M)$ and $S(M)$ in the case where $f$ is clear from the context.

	For an infinite $\omega$-limit set $\omega_f(x)$, let
	\[\mathcal{C}(x)=\{X\colon X\subset G \text{ is a cycle of graphs and } \omega_f(x)\subset X\}.\]
	One can show that $\mathcal{C}(x)$ is never empty.
	Following ideas of Blokh in their paper \cite{Snoha}, Ruette and Snoha proved that given $x\in G$ and a cycle of graphs $M$, if
	$\omega_f(x)=D(M)$, or $\omega_f(x)=S(M)$  , then the periods of cycles in $\mathcal{C}(x)$ are bounded from above and $M$ is minimal in $\mathcal{C}(x)$ with respect to inclusion.		
	Moreover, if $\omega_f(x)$ is solenoidal set then there exists a sequence $\left(M_n\right)_{n\geq 1}$ of cycles of graphs in $\mathcal{C}(x)$ with strictly increasing periods $\left(p_n\right)_{n\geq 1}$ such that, for all $n\geq 1$, $M_{n+1}\subset M_n$ and $\omega_f(x)\subset\bigcap_{n\geq 1}M_n$.
	Furthermore, for all $n\geq 1$, $p_{n+1}$ is a multiple of $p_n$ and every connected component of $M_n$ contains the same number (equal to $p_{n+1}/p_n\geq 2$) of components of $M_{n+1}$.

	\section{Recurrent continua and omega limit sets}\label{sec:recurrent-vs-limitsets}

In this section we inspect the relations between the recurrent subgraphs and the $\omega$-limit sets of their elements.
	The main goal is to prove Lemma~\ref{lem:containscycles}.

 While this is not directly stated in the assumptions, for each result in the rest of the paper we assume that we are given
a dynamical system $(G,f)$ on some arbitrary topological graph $G$. Let us also recall that we refer to Remark~\ref{rem:ambient}, which provides standing assumption on the metric on $G$
for the whole paper.

\begin{lemma}\label{lem:limitset}
If a continuum $A$ is recurrent then $A\cap \omega_f(x)\neq \emptyset$ for every $x\in A$.
\end{lemma}
\begin{proof}
Take some $x\in A$ and let $n_k$ be a sequence such that $\lim_{k\to \infty} d_H(A,f^{n_k}(A))=0$.
Then there are $y_k\in A$ such that $\lim_{k\to\infty} d(f^{n_k}(x),y_k)=0$.
Without loss of generality we may also assume that the limit $\lim_{k\to \infty} f^{n_k}(x)=z$ exists,
and therefore $\lim_{k\to \infty} y_k=z\in A\cap \omega_f(x)$.
\end{proof}

	\begin{lemma}\label{lem:contained-cir}
		If a continuum $A$ is recurrent and there is some circumferential set $\omega=S(K,f)$ such that $\omega_f(x)=\omega,$ for all $x\in A,$ then $A\subset K.$ 
	\end{lemma}
	\begin{proof}
		Take $y\in\omega$ not being a branching point nor endpoint of $K$ and some open neighborhood $U\subset K$ of $y,$ contained in the same edge as $y.$ 
		Then the family of preimages $\{f^{-n}(U)\}_{n \geq 0}$ is an open cover for $A.$ 
		But $A$ is compact so it admits finite subcover and since $K$ is invariant we can pick some $m\geq 0$ such that $f^m(A)\subset K.$
		But $A$ is recurrent and $K$ is closed, so the lemma is proved.
	\end{proof}

\begin{lemma}\label{lem:connectingcontinua}
Let $A$ be a recurrent subcontinuum and assume that there are continua $J_1,J_2\subset A$ which satisfy $f(J_i)\supset J_i$.
Then there is $k>0$ and a continuum $J_1\cup J_2\subset J \subset A$ such that $J\subset f^k(J)$.
\end{lemma}

\begin{proof}
	Denote $K_1=\overline{\bigcup_{m\geq 0} f^m(J_1)}$ and $K_2=\overline{\bigcup_{m\geq 0} f^m(J_2)}$
	Observe that $K_1, K_2$ are continua, $f(K_1)=K_1,$ $f(K_2)=K_2$ and $K_1\cup K_2\subset A$, since $A$ is recurrent
	and $\left(f^m(J_i)\right)_{m\geq 0}$ is a nested sequence for $i=1,2$.
	We distinguish two possible cases.
	If $K_1\cap K_2\neq\emptyset$ then it is enough to put
put $J=K_1\cup K_2$ and the proof is completed, so assume that it is not the case.

Since $K_1,K_2\subset A$ there is at least one arc $L_1\subset A$ such that $L_1\cap K_i\neq \emptyset$ and $\Int L_1\cap K_i= \emptyset$ for $i=1,2$.
	But by definition $K_1\cap f(L_1)\neq \emptyset$ and $K_2\cap f(L_1)\neq \emptyset$, so there is an arc $L_2\subset f(L_1)$ such that $L_2\cap K_i\neq \emptyset$ and $\Int L_2\cap K_i=\emptyset$.
	We continue with this process obtaining a sequence of arcs $L_i$ connecting $K_1$ with $K_2$ and such that $f(L_i)\supset L_{i+1}$. Then there are $k,m$ such that $f^k(L_m)\supset L_m$. 
	If we put $J=K_1\cup L_m\cup K_2$, then by definition $J\subset f^m(A)$ and $f^k(J)\supset J$ and we may assume (by taking larger $m$ and replacing $J$ by its iterate)
that $k$ divides $m$. Then $J\subset f^{jk+m}(A)$ for each $j\geq 0$, and since $A$ is recurrent for $f^k$, we must have $J\subset A$. The proof is completed.
\end{proof}

	In what follows, we consider the case when a recurrent continuum $A$ contains a point $x$ such that $\omega_f(x)$ is a subset of some basic set $D(X)$.
	In order to deal with basic sets, we will use the model map $g\colon Y\to Y$, where $Y$ is a cycle of graphs, $g$ being almost conjugate to $f\vert_{D(X)}$ and $g^n\vert_{Y_i}$ being mixing on every component $Y_i$ of $Y$. The reader is referred to \cite{B1,B2} for more details, see also \cite{Snoha,Forys}. 	

\begin{lemma}\label{lem:modelBS}
	Let $f\colon G\rightarrow G$ be a graph map and $X\subseteq G$ be a cycle of graphs. Suppose that $D(X)$ is a basic set. Then there is a transitive map $g\colon Y\rightarrow Y$, where $Y$ is a cycle of graphs $Y_0,\ldots Y_{n-1}$ with possibly non-empty intersection in the endpoints, and $\phi\colon X\rightarrow Y$ which almost conjugates $f|_{D(X)}$ and $g$. Moreover, $g^n|Y_i$ is mixing, for $i=0,\ldots,n-1$. The period $n$ of $Y$ is a multiple of the period of $X$ and $Y_i\cap Y_j=\End(Y_i)\cap \End(Y_j)\neq \emptyset$ iff $i\neq j$ and $i$ and $j$ are congruent modulo the period of $X$.	
\end{lemma}

\begin{lemma}\label{lem:basicset-sub} 
	Assume that a continuum $A$ is recurrent and $\omega=\omega_f(x)$ 
	is a basic set for some $x\in G$. 
	Denote by $K_1,K_2,...,K_m$ minimal (in the sense of inclusion) graphs such that $\omega\subset \bigcup_i K_i$, $f$ permutes graphs $K_i$
	and $m$ is such that $f^m|_{K_i}$ is almost conjugated to a mixing map.
	If $\Int A\cap \omega \cap K_i\neq\emptyset$ then $K_i\subset A$.
\end{lemma}

\begin{proof}
We will rely on the structure of basic sets from Lemma~\ref{lem:modelBS}.
	Since $f^m|_{K_i\cap\omega}$ is mixing and $\Int A\cap K_i$ is an open set intersecting $\omega$,
	 $A\cap K_i\cap \omega$ is an open set in relative topology of $\omega$ and hence
	 we have $\lim_{j\to \infty} d_H(K_i\cap \omega, f^{jm}(A\cap K_i\cap \omega))=0$. 
	 Since $A$ is recurrent set of $f^m$, for every $n>0$ there is arbitrarily large $j_n$ such that $f^{j_n m}(A)\subset N(A,1/n)$. By mixing (e.g. see \cite[Theorem~49]{POCoherent})
	 there is a dense set of points $Q$ in $\omega\cap K_i$ such that for an continuous function $F\colon Q\to \omega\cap K_i$ and any $q\in Q$ values $f^{ j_nm}(q)$ converge to $F(q)$ over a subsequence $\{j_n\}$. As a consequence
	 \[\omega \cap K_i=\omega_f(y)\cap K_i=\omega_{f^m}(y)\subset A,\]
	 therefore $K_i\cap \omega\subset A$.

 We claim that the following stronger condition $K_i\subset A$ holds.
	Let $g\colon Y\to Y$ be a mixing graph map such that there is a semi-conjugacy $\phi\colon K_i\to Y$ between $(K_i,f^m)$ and $(Y,g)$, almost conjugating $f^m\vert_{\omega\cap K_i}$ and $g$.
Suppose on the contrary, that $K_i$ is not a subset of $A$.
	Then there exists a connected open set $D_0$, a fiber of $\phi,$ not fully contained in $A.$ 
	Since the model map is mixing, there exists a sequence $(x_j)_{j\geq 0}$ in $Y$ of pairwise distinct points where $x_0=\phi(D_0)$ and $g(x_j)=x_{j-1}.$
	This way we construct a sequence of pairwise disjoint connected sets $D_j=\phi^{-1}(x_j)\subset K_i$ (generated by a backward extension of the trajectory of $D_0$) such that $f^m(D_j)=D_{j-1}$.
	By the definition of almost conjugacy, each $D_j$ intersects $\omega$ in its boundary.
	Since $A$ is connected, 

		$D_j\subset A$ for all $j$ sufficiently large. But then $D_0\subset f^{jm}(A)$ for all $j$ sufficiently large, which implies that $D_0\subset A$
		by the fact that $A$ is recurrent under $f^m$. A contradiction.
		This shows that $K_i\subset A$, completing the proof.
\end{proof}
\begin{corollary}\label{cor:basicset-sub} 

	 Assume that a nondegenerate continuum $A$ is recurrent and $\omega=\omega_f(x)$ 
is a basic set for some $x\in G$.
	Denote by $K_1,K_2,...,K_r$ minimal (in the sense of inclusion) graphs such that $\omega\subset \bigcup_i K_i$, $f$ permutes graphs $K_i$
	and $r$ is such that $f^r|_{K_i}$ is almost conjugated to a mixing map.
	If $A\cap\omega\cap K_i\neq\emptyset$ then either $K_i\subset A$
or  $A\cap\omega\cap K_i$ is finite set and each $y\in A\cap\omega\cap K_i$ is eventually periodic and 
moreover, $A\cap K_i\cap \omega$ contains a periodic point from orbit of each such $y.$
\end{corollary}

\begin{proof}
	If $\Int A\cap K_i\cap \omega\neq\emptyset$ 
then $K_i\subset A$ by Lemma~\ref{lem:basicset-sub},
	so suppose that $\Int A\cap K_i\cap \omega=\emptyset,$ i.e.
$A\cap K_i\cap \omega\subset \partial A$. 	 		In particular $K_i\setminus A\neq \emptyset$.
	The intersection $A\cap K_i\cap \omega$ is obviously finite so, to finish the proof, take some $y\in A\cap\omega\cap K_i$ 
	and let us prove that its orbit is finite.
Suppose on the contrary, that the trajectory of $y$ consists of pairwise distinct elements.
Since $A\in\Rec(\widetilde{f^r}),$ we may choose increasing sequence $(n_k)_{k\geq 1}$ of positive integers such that $\dist_H(A,f^{rn_k}(A))\to 0.$	
	In particular, we may assume that  $\diam f^{n_k}(A)>\diam A/2$ for all $k\geq 1$.
	We may also assume that the limit $\hat{y}=\lim_k f^{rn_k}(y)\in \omega\cap K_i$ exists and that the sequence $f^{rn_k}(y)\in \omega\cap K_i$ is monotone in an edge of $G$.
Since $f^{rn_k}(y)$ are all distinct and each $f^{rn_k}(A)$ is 	connected, we have that for some $k_0$,
$\left(\Int f^{rn_{k_0}}(A)\right)\cap\omega\cap K_i\neq\emptyset$ and hence $K_i\subset f^{rm}(A)$ for each $m\geq n_{k_0}$ by Lemma~\ref{lem:basicset-sub}, contradicting the recurrence of $A$.	
	We conclude that $\mathcal{O}(y)$ is finite, hence contains a periodic point and $y$ is eventually periodic. By Lemma~\ref{lem:limitset} applied to $A$ and $f^r,$ $A$ contains a periodic point from 
  $\mathcal{O}(y)\cap K_i.$
 \end{proof}

	After dealing with basic sets, in the following two lemmas we consider the case when a nondegenerate recurrent continuum $A$ contains a point $x$ such that $\omega_f(x)$ is a subset of a maximal solenoid $\omega$.
	Recall that, unlike for the other types of maximal $\omega$-limit sets, here we can not point out minimal cycle of graphs containing $\omega$.
	Still, we will show that that we can always pick some cycle of graphs $K\supset\omega$ such that, for each its component $K_i$ we either have $A\cap\omega\cap K_i=\emptyset$ or $K_i\subset A$.
	\begin{lemma}\label{lem:solenoid_1}
Let $A$ be a nondegenerate recurrent continuum such that for some $x\in A$ the $\omega$-limit set $\omega_f(x)$ is contained in a maximal solenoid $\omega$. 
	Let $(K^n)_{n\geq 1}$ be a nested sequence of cycles of graphs containing $\omega$ with strictly increasing periods $p_n$.
	 Then there is some cycle of graphs $K^m$ from that sequence with some component contained in $A.$ 
\end{lemma}
\begin{proof}
	First, we are going to show that there is some cycle of graphs $K^m$ from the sequence with period $p^m\geq 1$ and some $s\geq 0$ such that  $K^m_j\subset f^s(A)$ for some component  $K^m_j$ of $K^m.$ 
	
	By Lemma~\ref{lem:limitset} we have $A\cap\omega\neq \emptyset$, so let us fix a point $x\in A\cap\omega$
and let $n_k$ be an increasing sequence such that $\lim_{k\to \infty} d_H(A,f^{n_k}(A))=0$, in particular $\diam f^{n_k}(A)>\diam A/2$ for every $k$.
We may also assume that there is a free arc $J$ such that $y_k:=f^{n_k}(x)\in J$ is a monotone sequence in the ordering of $J$.
Since $\omega$ does not contain periodic points, all $y_k$ are different.

	But if we take sufficiently large $r$ and denote $s=n_r$ then there is $k$ such that $f^s(A)\ni y_k,y_{k+1},y_{k+2}$.
If we fix a cycle of graphs $K^m=\cup_{i=0}^{p_m-1} f^i(K^m_0)\supset \omega$ of period $p_m>n_{k+2}$ then by definition $y_k,y_{k+1},y_{k+2}$ belong to different elements $f^i(K^m_0)$ of the cycle.
In particular there is $j$ such that $y_{k+1}\in f^j(K^m_0)=K^m_j\subset f^s(A)$, and so $K^m_j$ is a subarc of $J$. 
Taking a few more iterations of $A$, we find $r$ and $i$ such that $K_i^m=f^i(K^m_0)\subset f^{rp_m}(A)$. 
Since $A$ is recurrent under $f^{p_m}$ and, 
	for all $l\geq 0$,
\[
K^m_i=f^{i}(K^m_0)=f^{i+lp_m}(K^m_0)\subset f^{(r+l)p_m}(A),
\] 
the result follows.
\end{proof}

\begin{lemma}\label{lem:solenoid-components}
	If a nondegenerate continuum $A$ is recurrent and there is $x\in A$ such that $\omega_f(x)$ is contained in a maximal solenoid $\omega,$ then there is a cycle of graphs $K$ containing $\omega$ with the property that for each $y\in\omega\cap A,$ $A$ contains a component of $K$ to which $y$ belongs.
\end{lemma}

\begin{proof}
	Fix any $x\in A$ and a maximal solenoid $\omega\supset \omega_f(x)$. There is $z\in G$ such that $\omega= \omega_f(z)$.

%

Fix any $y\in A\cap\omega$ and suppose that there exists a nested sequence $\{K^n\}_{n\geq 0}$ of cycles of graphs with increasing sequence of periods $p_n$ such that $\omega\subset K^n$  and, for each $n$, the unique component $K^n_{j^n_y}\ni y$ of $K^n$ is not fully contained in $A$, i.e. $K^n_{j^n_y}\setminus A\neq \emptyset$. 
	Note that, for each $n$, $y\in \omega\cap K^n_{j^n_y}=\omega_{f^{p_n}}(f^{j^n_y}(z))$, $\omega_{f^{p_n}}(f^{j^n_y}(z))$ is a solenoid for $f^{p_n}$ and $A$ is recurrent under $f^{p_n}$. Therefore, if we fix a sequence $(K^m\cap K_{j_y^n}^n)_{m>n}$ of cycles of graphs for $f^{p_n}$, by Lemma~\ref{lem:solenoid_1} there is a component $K^s_{j^s_A}$ of  a smaller cycle of graphs { $K^s$ (for some $s=s(n)>n$)} from the starting sequence such that  $K^s_{j^s_A}\subset A\cap K^n_{j^n_y}$.
%
 Using this procedure, we construct strictly increasing sequence of naturals $(n_k)_{k\geq 1}$ as follows.
	 We start with $n_1=1$, use $n=1$ in the previous paragraph, fix a component $K^1_{i^1_y}\ni y$ and, by setting $n_2\coloneqq { s(n_1)}$, we find a cycle of graphs $K^{n_2}$ with component $K_{i^{2}_A}^{n_2}$ such that $K_{i^{2}_A}^{n_2}\subset A\cap K^1_{i^1_y}\ni y$.
	 We repeat the procedure inductively by setting $n\coloneqq n_{k}$ in step $k+1$, obtaining $n_{k+1}\coloneqq {s(n_k)}$.
This way we obtain decreasing sequence $K^{n_k}$ of cycles of graphs containing $\omega$ such that for each element of the sequence, there is a component $K^{n_k}_{i^k_y}$ containing $y$ such that $K^{n_k}_{i^k_y}\setminus A\neq \emptyset$ and a component $K^{n_k}_{i^k_A}\subset A.$ 
	Also, the sequence satisfies $K^{n_{k+1}}_{i^{k+1}_A}\subset K^{n_k}_{i^k_y}$.

Observe that for any edge $E,$ if $E\subset K_{i^k_y}^{n_k}$ for all $k\geq 0$ then obviously $E$ does not intersect $K_{i^k_A}^{n_k}$ for any $k$ since different components of the same cycle are disjoint.
	Also, if $E\cap K_{i^{k_0}_y}^{n_{k_0}}=\emptyset$ for some $k_0,$ then $E\cap K_{i^k_y}^{n_{k}}=\emptyset$ holds for each $k\geq k_0.$
	Since there are finitely many branching points, all but finitely many elements of the family $\{K_{i^k_A}^{n_k}\}_{k\geq 0}$ are subsets of the edges which are intersected with but not contained in $\bigcap_{k\geq 0}K_{i^k_y}^{n_k}.$ 
	We conclude that there has to be an edge $E_0$ and some $k,r\geq 1$ such that
	
	\begin{enumerate}
\item\label{con1:definingK} there is a point $q\in E_0\setminus K^{n_{k-1}}_{i^{k-1}_y}$ and there is a component $K^{n_k}_{i^k_A}$ of  $K^{n_k}$ contained in $A\cap E_0$, 
\item\label{con2:definingK} $K^{n_{k+r}}$ has a component $K^{n_{k+r}}_{i^{k+r}_A}\subset A$ such that $K^{n_{k+r}}_{i^{k+r}_A}$ and $K^{n_k}_{i^k_A}$ are contained in the same component $C$ of  $K^{n_{k-1}}_{i^{k-1}_y} \cap E_0\setminus\{y\}$ (whether $y\in E_0$ or not).
\end{enumerate}
When \eqref{con1:definingK} is satisfied, to ensure \eqref{con2:definingK} it is enough to consider a subsequence of $n_k$ such that all sets $K^{n_{k}}_{i^k_A}$
are contained in the same component of $E_0\setminus\{y\}$.
	Note that, from the construction, $K_{i^{k+r}_A}^{n_{k+r}}\subset K^{n_k}_{i^k_y}$ while $K_{i^k_A}^{n_{k}}\cap K^{n_k}_{i^k_y}=\emptyset.$

We claim that $K_{i_A^{k+r}}^{n_{k+r}}$ lies between $K_{i_A^{k}}^{n_{k}}$ and $K_{i_y^{k+r}}^{n_{k+r}}\cap C$ in the ordering of $E_0.$
To prove the claim, 
first, notice that $K_{i_A^{k}}^{n_{k}}$, $ K_{i_A^{k+r}}^{n_{k+r}}$ and $K_{i_y^{k+r}}^{n_{k+r}}\cap C$ are all subintervals of $C$ since $E_0\setminus C\neq\emptyset.$
	Moreover, those three sets are pairwise disjoint by the assumptions.
	If $K_{i_y^{k+r}}^{n_{k+r}}\cap C$ would lie between the other two then we would have $ K_{i_y^{k+r}}^{n_{k+r}}\cap C=K_{i_y^{k+r}}^{n_{k+r}}\ni y$ which would contradict 
	the definition of $C$.
Similarly, $K_{i_y^{k}}^{n_{k}}\cap C\subset C\setminus K_{i_A^{k}}^{n_{k}}$ is necessarily connected set containing $ K_{i_A^{k+r}}^{n_{k+r}}$ and $K_{i_y^{k+r}}^{n_{k+r}}\cap C$ so $K_{i_A^{k}}^{n_{k}}$ cannot lie between $K_{i_A^{k+r}}^{n_{k+r}}$ and $K_{i_y^{k+r}}^{n_{k+r}}\cap C$ and the asserted ordering follows.
	 Indeed, the claim holds.
 
	Let arc $L$ be the union of $K^{n_k}_{i^k_A}\cup K_{i^{k+r}_A}^{n_{k+r}}$ and the unique arc in $C\subset E_0$ connecting those two sets.
	We can also assume that $L\subset A$ since $A\cap E_0$ can have at most two components and we can choose infinitely many such pairs $K_{i^{k+r}_A}^{n_{k+r}},K_{i^{k}_A}^{n_k}$ as subsets of some edge. 
	The map $f^{p_{n_k}}$ fixes $K^{n_{k-1}}_{i^{k-1}_y}$ and $K^{n_k}_{i^k_y}$ but $K^{n_{k+r}}$ has period $p=\frac{p_{n_{k+r}}}{p_{n_k}}\in\mathbb{N}\setminus\{1\}$ under $f^{p_{n_k}}.$ 

	There is some $m\leq p$ such that $K^{n_{k+r}}_{i^{k+r}_y}\subset f^{mp_{n_k}}(L).$ 
Furthermore, $K^{n_{k}}_{i^{k}_A}=f^{mp_{n_k}}(K^{n_{k}}_{i^{k}_A})\subset f^{mp_{n_k}}(L)$ and $f^{mp_{n_k}}(L)\subset f^{mp_{n_k}}(K^{n_{k-1}}_{i^{k-1}_y})=K^{n_{k-1}}_{i^{k-1}_y}\not\ni q$, 
	 and $K_{i_A^{k+r}}^{n_{k+r}}$ lies between $K_{i_A^{k}}^{n_{k}}$ and $K_{i_y^{k+r}}^{n_{k+r}}\cap C$ in ordering of $E_0$, so since $f^{mp_{n_k}}(L)$ is connected, 
we must have $L\subset f^{mp_{n_k}}(L)$
	and as a consequence $K^{n_{k+r}}_{i^{k+r}_y}\subset f^{jmp_{n_k}}(A)$ for every $j>1$. This contradicts recurrence of $A$
	under $f^{mp_{n_k}}$ because by definition $K^{n_{k+r}}_{i^{k+r}_y}\setminus A\neq \emptyset$.
	We have just proved that for any {$y\in A\cap\omega$} there exists a cycle of graphs $K$ such that  $\omega\subset K^n$  and the unique component $K_{i}\ni y$ of $K$ satisfies $y\in K_i\subset A$.

	For each edge $E$ of $G$ such that $E\setminus A\neq\emptyset$ and each component $B$ of $E\cap A,$ if  $C=B\cap \omega\neq \emptyset$ then $C$ is a closed nonempty subset of $E,$ so there exist minimal and maximal points $y^C_1,y^C_2$ of $C$ in the ordering of $E.$ 
	Fix any edge $E$ of $G$ such that $E\setminus A\neq\emptyset$ and associated sets $B,C$ as above.
	
	By previous arguments there is a cycle of graphs $K$ containing $\omega$ and we can select its components $y_i^C\in K_i\subset A$.
	Then any $y\in C$ is contained in a component of $K$ intersecting $B$. But then it is one of $K_i$ or a component of $K$
	between them, since components of $K$ are pairwise disjoint. In any case, it belongs to a component of $K$ fully contained in $B\subset A$, because $y_i^C$
	are extreme points of $B\cap \omega$. This property is hereditary by any cycle of graphs $\hat{K}$ containing $\omega$ such that $\hat{K}\subset K$.
	
	The above argument about $C$ can be repeated with $B=E$ for any edge $E\subset A$ such that $\Int E\cap \omega\neq \emptyset$.
	
	Since we have finitely many sets $C$, we have also finitely many points $y^C_i$, and so the result follows by an easy induction.

\begin{figure}[H]

\includegraphics[scale=0.5]{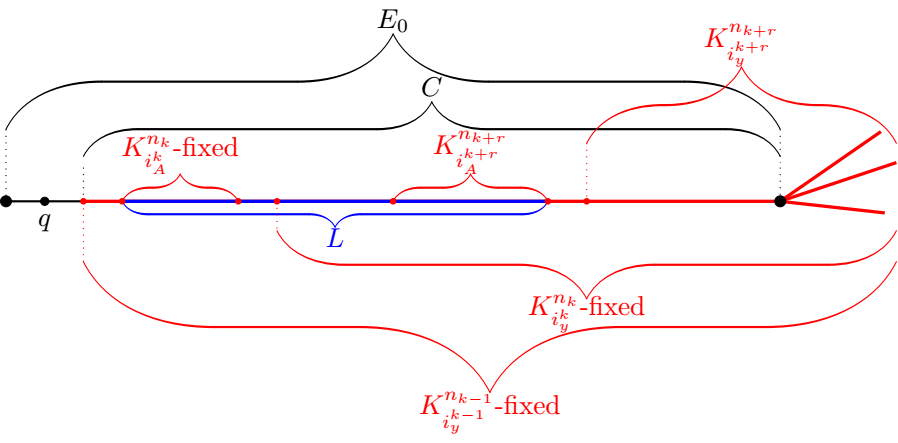}
\caption{Sketch for the case $y\notin E_0$: $K^{n_{k+r}}_{i^{k+r}_y}$ is not necessarily an arc or a subset of $E_0$. $L$ is subset of $K^{n_{k-1}}_{i^{k-1}_y}$ which is invariant under $f^{p_{n_k}}$ and does not contain $E_0$ so $L$ has to grow. Also, the order of $K_{i_A^{k+r}}^{n_{k+r}}$, $K_{i_A^{k}}^{n_{k}}$ and $K_{i_y^{k+r}}^{n_{k+r}}\cap C$ in $E$ is shown.}\label{fig:solenoid}
\end{figure}
\end{proof}
	The following three lemmas cover the case when a nondegenerate recurrent continuum $A$ contains a point $x$ such that $\omega_f(x)$ is a subset of some circumferential set $S(X)$.
	Since $S(X)$ is always a minimal set, we necessarily have that $\omega_f(x)=S(X)$.
	The reader not familiar with structure of graph maps without periodic points is referred to \cite{Shao}.
	Note that Lemma~\ref{lem:contained-cir}	gives us the right to partition the set of all nondegenerate recurrent subcontinua into two disjoint sets.
	The one contains continua whose each point has the same circumferential set as their $\omega$-limit set and the other contains all the remaining continua.
	In what follows, we deal with the latter type of subcontinua, while the former will be considered in Lemma~\ref{lem:recurrent:circ}.

\begin{lemma}\label{lem:periodic-point}
If a nondegenerate continuum $A$ is recurrent and there is $x\in A$ such that $\omega_f(x)$ is not a circumferential set, then $A$ contains a periodic point.
\end{lemma}
\begin{proof}
	By Lemma~\ref{lem:limitset}, there is $y\in\omega_f(x)\cap A.$
	If $\omega_f(x)$ is finite and therefore a periodic orbit then there is nothing to prove.
	
	Let $\omega_f(x)$ be some infinite subset of a basic set $D(K,f)$.
	 Let $C\ni y$ be one of graphs described in Lemma~\ref{lem:basicset-sub}, subgraph of some component of $K$, fixed by $f^r$.
 If $\omega_f(x)\cap \Int A\neq \emptyset$ is infinite then Lemma~\ref{lem:basicset-sub} gives that $C\subset A$ so $A$ contains a periodic point in this case.
	In the opposite case $A$ contains a periodic point by Corollary~\ref{cor:basicset-sub}.
	Finally, let $\omega_f(x)$ lie in some maximal solenoid $\omega.$
	By Lemma~\ref{lem:solenoid-components}, there is a component $y\in C\subset A$ of some cycle of graphs with period $p$ containing $\omega.$
	We claim that $C$ contains a periodic point.
	Indeed, since $C$ contains nested sequence of cycles of graphs for $f^p$ containing $\omega$ with higher periods and each cycle has pairwise disjoint components, some of those components are intervals and hence contain a periodic point by fixed-point property.
 \end{proof}

\begin{lemma}\label{lem:growing-arc}
If a nondegenerate continuum $A$ is recurrent and there are $x_1,x_2\in A$ such that $\omega_f(x_1)=S(K,f)$ is circumferential set and $\omega_f(x_2)\neq \omega_f(x_1)$ then there are some $s\geq 0,t\geq 1$ and some arc $L\subset f^s(A)$ intersecting $K$ such that $f^t(L)\supset L.$
\end{lemma}

\begin{proof}
	If $\omega_f(x_2)$ is a circumferential set, then
	let $K^\prime$ be a minimal cycle of graphs containing $\omega_f(x_2).$ 
	Since $\omega_f(x_1)$ is the unique minimal set in $K$ and $\omega_f(x_2)$ is the unique minimal set in $K^\prime,$ we have that $K\cap K^\prime=\emptyset.$ 
	If $\omega_f(x_2)$ is not circumferential, then by Lemma~\ref{lem:periodic-point}, $A$ contains a periodic point. 
	Denote its orbit by $K^\prime.$ 
	Obviously, in this case we also have $K\cap K^\prime=\emptyset.$  	
	Let us enumerate the connected components of $K$ by $K_i,$ $i=1,...,r$ and of $K^\prime$ by $K_i^\prime,$ $i=1,...,r^\prime.$
	By Lemma ~\ref{lem:limitset}, we can assume that $A$ intersects $K_1$ and $K_1^\prime.$
	Consider the finite family $\mathcal{L}$ of all arcs contiguous to $K_1$ and $K_1^\prime.$
	There is some $L_0\in\mathcal{L}$ such that $L_0\subset A,\ L_0\cap K_1=\{y\},\ L_0\cap K_1^\prime=\{y^\prime\}.$
	Since $f^{r^\prime r}(y)\in K_1$ and $f^{r^\prime r}(y^\prime)\in K_1^\prime,$ there is some $L_1\in\mathcal{L},$ $L_1\subset f^{rr^\prime}(L_0).$
	This way, for each $k\geq 1,$ we obtain some $L_k\in \mathcal{L},$ $L_k\subset f^{rr^\prime}(L_{k-1})\subset f^{krr^\prime}(A).$
	Since $\mathcal{L}$ is finite, there have to be some $s,t,$ multiples of $rr^\prime$ and some $L\in\mathcal{L},$ $L\subset f^s(A)$ such that $f^t(L)\supset L.$ 
\end{proof}

\begin{lemma}\label{lem:circumferential}
If a nondegenerate continuum $A$ is recurrent and there is $x\in A$ such that $\omega_f(x)\neq S(K,f),$ where $S(K,f)$ is a circumferential set 
then $C\subset A$ for each component $C$ of $K$ intersecting $A.$
\end{lemma}
\begin{proof}
	Take some component $C$ of $K$ such that there is $x^\prime\in A\cap C.$ 
	We will show that $C\subset A.$
	If $K$ has period $r,$ then $\omega_{f^r}(x^\prime)=S(K,f)\cap C=S(C,f^r).$ 
	In particular, $\omega_{f^r}(x)\neq\omega_{f^r}(x^\prime)$ so we can apply Lemma~\ref{lem:growing-arc}.
	Hence, there are some $s\geq 0, t\geq 1,$ multiples of $r,$ and some arc $L\subset f^s(A)$ intersecting $C$ such that $f^t(L)\supset L.$ 
	 
	Take any $y\in L\cap C.$ 
		Since 
$\omega_{f}(x')\cap C$ is the unique minimal set of $(C,f^r|_{C})$~(e.g. see \cite{Shao}),  	it is totally minimal, $r$ divides $t$ and $(C,f^r|_{C})$
is almost 1-1 extension of an irrational rotation,
we have that 
\[
\omega_{f^t}(y)=\omega_{f^r}(y)=\omega_f(x')\cap C.
\]
	Since the above intersection is infinite, there is some free arc $J\subset C$ and a sequence $y_1,y_2,y_3\in\omega_f(x^\prime)\cap\Int J$, monotone in the ordering of $J.$ 
	Let $U_1,U_2,U_3$ be pairwise disjoint connected open neighborhoods (that is, open intervals) of $y_1,y_2,y_3$ in $J,$ respectively. 
	There are $m_1,m_2,m_3,$ such that $f^{tm_i}(y)\in U_i,$ for $1\leq i\leq 3.$ 
	For $m=\max\{m_1,m_2,m_3\},$ since consecutive iterations of $L$ form a nested sequence:
\begin{equation}
L\subset f^t(L)\subset \ldots \subset f^{mt}(L)\subset \ldots\label{eq:recLnest}
\end{equation}
and 
	therefore $f^{mt}(L)\cap U_j\neq \emptyset$ for $j=1,2,3$.
	 Since $f^{mt}(L)$ is connected we immediately have that $U_j\subset \Int f^{t(m+i)}(L)$ 
	 for at least one $1\leq j\leq 3$ and all $i\geq 0$.
	 For the above $j$ there is $l$ such that $\cup_{i=0}^{l}f^{ti}(U_j)\supset \omega_{f^t}(y)=S(C,f^t)$, and since $\cup_{i=0}^{l}f^{ti}(U_j)$ possibly does not cover at most finitely many fibers of the associated 
	 	almost 1-1 extension $(C,f^r|_{C})$ of an irrational rotation, we obtain that $C\subset f^{ti}(L)$ for all sufficiently large $i$.	 	
 By recurrence of $A$, we obtain that $C\subset f^s(A)$
	
	Since $A$ is recurrent under $f^r$, $r$ divides $s$ and for each $n\geq 0$ we have $C=f^{nr}(C)\subset f^{s+nr}(A)$,  we conclude $C\subset A.$
\end{proof}

\begin{lemma}\label{lem:similarimages}
Let $A\in \Rec(\tilde{f})$ be a nondegenerate continuum. Then there exist some $k\geq 0$ and some $\epsilon>0$ such that for each $m\geq k,$ the condition $\dist_H(f^m(A),f^k(A))<\epsilon$ implies following two properties: \\
Firstly, that  for each edge $E$ of $G$, each its endpoint $e$ and
for each side $T$ of $e$ with elements contained in $E,$
the following two statements are equivalent:
\begin{enumerate}[(1)]
\item there is a one-sided neighborhood $V_k\in T$ with $V_k\subset f^k(A)$;
\item there is  $V_m\in T$ with $V_m\subset f^m(A).$
\end{enumerate}  
Secondly, for each edge $E$ of $G$, if $C$ is a component of $\Int E\cap f^k(A)$ and $D$ is a component of $\Int E\cap f^m(A)$ and  $C\cap D\neq \emptyset$
then $D$ is a unique component of $\Int E\cap f^m(A)$ that $C$ intersects and vice-versa.
\end{lemma}

\begin{proof}
Pick some edge $E_1$ of $G.$ 
	Let $\epsilon_1^\prime$ be the minimal value among the lengths of components of $\Int E\cap A$ and of $E\setminus A$ (we take into account only nonempty ones).
Note that if $\dist_H(A,f^p(A))<\frac{1}{2}\epsilon_1^\prime,$ for some $p>0$ and $E_1\subset f^p(A)$ then we must also have $E_1\subset A$.

	If for each $n$ such that $\dist_H(A,f^n(A))<\frac{1}{2}\epsilon_1^\prime,$ 
all the conditions from the statement regarding $E_1$ are valid then set $\epsilon_1=\epsilon_1^\prime$ and we proceed to the next edge.
	Now suppose there is some $p$ such that $\dist_H(A,f^p(A))<\frac{1}{2}\epsilon_1^\prime,$ but some of the conditions are not fulfilled.
	This is only possible if  one of the following cases arose:
	\begin{enumerate}
\item $E_1\subset A$ and $E_1\setminus f^p(A)\neq\emptyset.$ 
\item $E_1\setminus A\neq\emptyset,$  $\Int E_1\cap A $ has one component and $\Int E_1\cap f^p(A)$ has two components.
\item $\Int E_1\cap A=\emptyset$ and $\Int E_1\cap f^p(A)\neq\emptyset.$
 \end{enumerate}
	
	Repeat the process with $f^p(A)$ instead of $A$ and with new $\epsilon_1^{\prime\prime}<\frac{\epsilon^\prime_1}{2}$ calculated analogously.    	
	If for each $n\geq p,$ $\dist_H(f^p(A),f^n(A))<\frac{1}{2}\epsilon_1^{\prime\prime}$ implies that 
	all the conditions from the statement regarding $E_1$ are valid then set $\epsilon_1=\epsilon_1^{\prime\prime}$ and we proceed to the next edge.
	The opposite is only possible if there is some $r\geq p$ such that { $\dist_H(f^p(A),f^r(A))<\frac{1}{2}\epsilon_1^{\prime\prime}$, while}:
\begin{enumerate}
\item $\Int E_1\cap A=\emptyset$ and $\Int E_1\cap f^p(A)$ has one component and $\Int E_1\cap f^r(A)$ has two components, or
\item {$E_1\subset A,$} $\Int E_1\cap f^p(A)$ has one component and $\Int E_1\cap f^r(A)$ has two components.
\end{enumerate}	

	In the next iteration of this procedure, using analogous $\epsilon_1<\frac{\epsilon_1^{\prime\prime}}{2},$ the conditions for edge $E_1$ are certainly fulfilled.
	Now we proceed with another edge $E_2$ but starting with the last obtained iterate of $A$ and $\epsilon_2<\frac{\epsilon_1}{2}$ calculated analogously. 
	We are done after finitely many steps by setting $\epsilon=\frac{1}{2}\epsilon_s$ for last considered edge $E_s$ and $k$ obtained from the last observed iterate of $A.$
\end{proof}

\begin{lemma}\label{lem:containscycles}
Let $A\in \Rec(\tilde{f})$ be a nondegenerate continuum. For each $x\in A$ let $M_x\supset \omega_f(x)$ be a maximal $\omega$-limit set. If $M_x$ is not circumferential set for some $x$
	or there are $x,y\in A$ such that $M_x\neq M_y$ then there exist $m,k\geq0$ such that $\omega_{f^m}(x)\subset f^k(A)$ for each $x\in f^k(A).$
Even more, $\bigcup_{x\in f^k(A)}\omega_{f^m}(x)$ is closed and for each $x\in f^k(A),$ $f^k(A)$ contains 
some periodic points of $\tilde{f}^m$, containing
$\omega_{f^m}(x).$
\end{lemma}

\begin{proof}
	Let $\epsilon>0$ and $k\geq 0$ be provided by Lemma~\ref{lem:similarimages}. Clearly $f^k(A)\in \Rec(\tilde{f})$.
	For each edge $E$ of $G$ such that $E\cap f^k(A)\cap\omega(f)\neq\emptyset$ and $E\setminus f^k(A)\neq\emptyset$ and for each component $D$ of $E\cap f^k(A)$ such that $D\cap\omega(f)\neq\emptyset$ let $y_D^i,$ $i\in\{1,2\}$, be minimal and maximal points (possibly $y_D^1=y_D^2$) of $\omega(f)\cap D$ in the ordering of $D$, i.e. each $y_D^i$ is at the boundary of the convex hull of $\omega(f)\cap D$.
	Denote by $\mathcal{D}$ the family of all such components $D$, and let $Y_0$ be the set of all such points $y_D^i$ (over all components $D\in \mathcal{D}$). 
	We distinguish two cases.

	First, suppose that for each $y\in Y_0$ and $z$ such that $y\in\omega_f(z),$ there is a graph $C\ni y,$ $C\subset f^k(A),$ being a periodic point of $\tilde{f}$, whose orbit contains $\omega_f(z)$.
	For technical reasons, we consider periodic orbit as (degenerated) cycle of graphs. 
%
	Denote by $p$ the least common multiple of the periods of those graphs.
	Since $f^k(A)$ is recurrent under $f^p,$ there is some $m\geq 1,$ such that $p\vert m$ and $\dist_H(f^k(A),f^{k+m}(A))<\epsilon.$
	By the virtue of construction, we have $\omega(f)\cap f^k(A)\subset\omega(f)\cap f^{k+m}(A)$ because $C\subset f^k(A)\cap f^{k+m}(A)$ for each periodic continuum $C$ considered for each $y_D^i$ before and
	by the definition of $k,\epsilon$ provided by Lemma~\ref{lem:similarimages}.

	Now suppose that there is some $\emptyset\neq Y_0^\prime\subset Y_0$ such that, for each $y\in Y_0^\prime$, $f^k(A)$ does not contain any graph periodic under $\tilde{f}$ and containing $y$.
	Fix any $y\in Y_0^\prime.$ By definition, $y$ is not a periodic point.	
	Observe that $y$ is an eventually periodic element of some basic set $B$.
	Indeed, $y$ has to be an element of some basic set $B$ by Lemmas~\ref{lem:circumferential} and~\ref{lem:solenoid-components} and additionally,
	by Corollary~\ref{cor:basicset-sub}, it has finite orbit.
	This shows that $y$ is an eventually periodic point, that is, it is eventually mapped onto a periodic point $y^\prime\in C_y\cap\partial f^k(A)$,
	 where $C_y\ni y$ is a subgraph described in Lemma~\ref{lem:basicset-sub}, periodic under $\tilde{f}$ with period $r$.
	 (i.e. $f^r|_{C_y}$ lifts to a mixing map by almost conjugacy).
	We also have that $\vert f^k(A)\cap C_y\cap B\vert>1$ since $y$ itself is not periodic.
If $(\partial f^k(A))\setminus (C_y\cap B)\neq\emptyset$ then denote $\epsilon_0^y=\frac{1}{2}\min\left\{\epsilon,\dist_H((\partial f^k(A))\setminus (C_y\cap B) , C_y\cap B)\right\},$ otherwise $\epsilon_0^y=\frac{1}{2}\epsilon.$ 
	Observe that if $z\in f^k(A)\cap C_y\cap B$ and for some $j\geq 0$ we have $f^{k+j}(z)\in C_y\cap B$ and $\dist(f^{k+j}(z),f^k(A))<\epsilon_0^y$ then $f^{k+j}(z)\in C_y\cap \partial f^k(A)$.
By definition $Y_0^\prime$ is finite, so we may define $\epsilon_0=\min\{\epsilon_0^y\vert y\in Y_0^\prime\}.$
	
	Denote by $p$ the least common multiple of the periods of 
 periodic points of $(C(G),\tilde{f})$, which are selected for points in $Y_0\setminus Y_0^\prime$ contained in $A$ and of the periods of periodic points from $\partial f^k(A)$.
	{ Note that if $B$ is a basic set with $B\cap f^k(A)\subset\partial f^k(A)$ then $p$ 
	 may be smaller than period of graphs described in Lemma~\ref{lem:basicset-sub} and containing $B$ (e.g. there is a point from $Y_0^\prime$ in the intersection of those subgraphs periodic under $\tilde{f}^p$).
	 If it is the case, we change $p$ by its multiple, so that it becomes a multiple of a period of mentioned graphs, periodic under $\tilde{f}$ and containing $B$.}
	Again, since $f^k(A)$ is recurrent under $f^p,$ we may choose some $m\geq 1,$ such that $p\vert m$ and $\dist_H(f^k(A),f^{k+m}(A))<\epsilon_0.$
	If we have 
	\begin{equation}
	\omega(f)\cap f^k(A)\subset\omega(f)\cap f^{k+m}(A),\label{includion:lemma:k+m}
	\end{equation} we are done with this part of the proof.

	Now suppose that the inclusion \eqref{includion:lemma:k+m} does not hold.
	By Lemma~\ref{lem:similarimages} and definition of $\epsilon_0$, there is 1-1 correspondence between connected components of $f^k(A)$ and $f^{k+m}(A)$ in each edge of the graph, therefore there is some $y\in Y_0^\prime$ such that $y\in \omega(f)\cap f^k(A)\setminus\omega(f)\cap f^{k+m}(A)$.
	Let $C\ni y$ be a subgraph described in Lemma~\ref{lem:basicset-sub}, periodic under $\tilde{f}$ with orbit containing $B$ and let $r$ be its period.
	We claim that $f^m(C)=C$.
Indeed, since $f^k(A)\cap B\cap C\neq\emptyset$ and $C\not\subset f^k(A)$ and since $f^k(A)\in\Rec(\widetilde{f^r}),$ Corollary~\ref{cor:basicset-sub} implies that $\partial f^k(A)$ contains some periodic point from $C.$
	By definition, $f^m$ fixes $C$, so the claim holds.
	Observe that, $(\Int f^{k+m}(A)) \cap C\cap B\neq \emptyset$
		would lead to a contradiction with recurrence of $f^k(A)$.  Namely, Lemma~\ref{lem:basicset-sub} would imply that for each $n\geq 1,$ $C\subset f^{k+nm}(A)$ and then $C\subset f^k(A)$ which is impossible.
	For a periodic point $z\in\partial f^k(A)\cap C$ we also have $z\in\partial f^{k+m}(A)$ since 
	$z$ is fixed by $f^m$.
	Note that $(f^k(A)\cup f^{k+m}(A))\cap C\cap B\subset \partial f^k(A)\cup  \partial f^{k+m}(A)$ and if  there existed $z\in (\partial f^{k+m}(A)\cap C\cap B)\setminus(\partial  f^{k}(A)\cap C)$ then we would have $\dist(z,f^k(A))\geq\epsilon_0$ and as a consequence $\dist_H(f^k(A),f^{k+m}(A))\geq\epsilon_0$ or we would have $z\in \Int f^k(A)$ which in both cases is a contradiction. 
This implies $f^{k+m}(A)\cap C\cap B\subset f^k(A)\cap C\cap B$,
which together with $y\in C\cap B\cap \partial f^k(A)\setminus f^{k+m}(A)$ leads to $\vert f^{k+m}(A) \cap C\cap B\vert<\vert f^k(A) \cap C\cap B\vert$.
		
	Let $D\in\mathcal{D}$ be such that $y=y_D^i$, for some $i\in\{1,2\}$.
	If $D\cap \omega(f)\cap f^{k+m}(A)\neq\emptyset,$ pick its new maximal point $y^\prime.$
	If it happens that $f^{k+m}(A)$ does not contain any subgraph, periodic under $\tilde{f}^m$ and containing $y^\prime$ then again $y^\prime\in B^\prime$ for some basic set $B^\prime.$
	Let $C^\prime\ni y$ be  subgraph described in Lemma~\ref{lem:basicset-sub} which is a periodic point of $\tilde{f}^m$ and whose $f^m$-orbit contains $B^\prime.$
	Observe that $C\neq C^\prime$ since  $\Int f^k(A)\cap C\cap B=\emptyset$.
	Now repeat the process with $Y_1,Y_1^\prime$ and $\epsilon_1<\frac{\epsilon_0}{2}$ calculated analogously.
	Observe that the starting condition provided by Lemma~\ref{lem:similarimages} is preserved.
	In finitely many steps of the procedure, we see that there are indeed some $k^\prime$ and $m^\prime$ such that $\omega(f)\cap f^{k^\prime}(A)\subset\omega(f)\cap f^{k^\prime+m^\prime}(A).$
	This is because a continuum in $G$ can have at most $2r$ boundary points, where $r$ is the number of edges in $G,$ and in each step points from $Y_{i+1}^\prime\setminus Y_i^\prime$ yield some periodic point from the boundary of our continuum which is also in the boundary of its further iterates.
	Also, by the way $\epsilon_i$ are picked, each point from $Y_{i+1}^\prime\setminus Y_i^\prime$ belongs to a basic set different from the basic sets corresponding to the points of $Y^\prime_i.$ 
	This completes the induction, obtaining at the end integers $k\geq 0,m\geq 1$ such that $\omega(f)\cap f^k(A)\subset\omega(f)\cap f^{k+m}(A)$. 

Now, let us prove the opposite inclusion, i.e. $\omega(f)\cap f^{k+m}(A)\subset\omega(f)\cap f^k(A).$
	Suppose on the contrary, i.e. that there is some $y\in \omega(f)\cap f^{k+m}(A)\setminus\left(\omega(f)\cap f^k(A)\right)$.
	Let $M_y$ be some maximal $\omega$-limit set containing $y.$
If $M_y$ is a circumferential set or a periodic orbit, then let $C$ be 
	 a minimal cycle of graphs (possibly degenerate) containing $M_y$
	 and if $M_y$ is a basic set, let $C$ be an orbit of a subgraph described in Lemma~\ref{lem:basicset-sub}.
	Note that, since $(f^m)^{-1}(y)\cap f^k(A)\neq\emptyset$, Lemma~\ref{lem:limitset} ensures $M_y\cap f^k(A)\neq \emptyset.$
	If $M_y$ is some maximal solenoidal $\omega$-limit set then 
	since $f^k(A)\cap M_y\neq\emptyset$, we can pick some cycle $C,$ containing $M_y$ and provided by Lemma~\ref{lem:solenoid-components} for $f^k(A)$.
		Then $C$ is provided for $M_y$ in any of possible cases.
	
	Let us enumerate the elements of the above orbit $C$, say $C_i,\ i=1,...,r$.		
	Suppose $y\in C_t.$ 
	We claim that $f^k(A)\cap C_t\cap M_y=\emptyset.$
	Indeed, if $M_y$ is maximal periodic orbit then it follows trivially.
	If $M_y$ is a circumferential set then, since {by assumptions} there is $z$ with $M_z\neq M_y$, we may apply Lemma~\ref{lem:circumferential} and therefore $C_t\cap f^k(A)\neq \emptyset$ would imply $y\in C_t\subset f^k(A)$ which is a contradiction with the choice of $y$.
If $M_y$ is solenoid then since $C$ was provided by Lemma~\ref{lem:solenoid-components}, we see that $f^k(A)\cap C_t\cap M_y=\emptyset$ as otherwise $y\in C_t\subset f^k(A).$
	Repeating previous arguments, we see that if  $f^k(A)\cap C_t\cap M_y\neq \emptyset$ and $M_y$ is a basic set, then we have a contradiction with the choice of $\epsilon_0$.
	Indeed, $f^k(A)\cap C_t\cap M_y=\emptyset$. 

	 Suppose (after renumbering components $C_i$) that $f^k(A)$ intersects $C_1\cap M_y,...,C_s\cap M_y$ for some $s$, 
$f^k(A)\cap\left(\bigcup_{i=s+1}^r C_i\right)\cap M_y=\emptyset$ and $t>s.$ 
	 Note that, since $\omega(f)\cap f^k(A)\subset\omega(f)\cap f^{k+m}(A)$, $f^{k+m}(A)$ intersects $C_1\cap M_y,...,C_s\cap M_y$.
	Clearly also $f^{k+m}(A)$ intersects $C_t$.
	But then for all $n\geq 1$, $f^{k+nm}(A)$ intersects $C_i\cap M_y$ for at least $s+1$ components $C_i$ of $C$, and this property extends onto $f^k(A)$ since it is a recurrent set.  
This contradicts the choice of $s$.
Indeed, 
\begin{equation}
\label{kmeqk} \omega(f)\cap f^k(A)=\omega(f)\cap f^{k+m}(A).
\end{equation}
	
	Now we prove $\omega=\bigcup_{x\in f^k(A)}\omega_{f^{m}}(x)\subset f^k(A).$
	Take some $y\in\omega_{f^m}(z),$ $z\in f^k(A).$
	By Lemma~\ref{lem:limitset}, we have that $\omega_{f^m}(z)\cap f^k(A)\neq\emptyset.$ If $\omega_{f^m}(z)$ is a periodic orbit then there is nothing to prove {by \eqref{kmeqk}}.
	Similarly, if 
	{$\omega_{f^m}(z)$} is a minimal solenoid or circumferential set then $\omega_{f^m}(z)\subset f^k(A)$. But then, Lemma~\ref{lem:solenoid-components}
	ensures that  $\omega_{f^m}(z)\subset f^k(A)$ also in the case that $\omega_{f^m}(z)$ is a solenoid with non-minimal points. The case of basic set intersecting interior of $f^k(A)$ is covered by Lemma~\ref{lem:basicset-sub}.
The only remaining possibility is that $\omega_{f^m}(z)\subset B$ where $B$ is a basic set such that $B\cap \Int f^k(A)=\emptyset$.
	 But Corollary~\ref{cor:basicset-sub} implies that $\omega_f(z)$ is a periodic orbit, and this case is already covered.
	 Indeed, $\omega\subset f^k(A).$
	
It is well known that  $\omega(f)$ is closed (e.g. see~~\cite{B1}). 
	We claim that $\omega$ is also a closed set. Fix any $z\in \omega(f)\cap f^k(A)\setminus \omega$. Note that $z\not\in \Int f^k(A)$ because then there is $y\in f^k(A)$ such that $z\in \omega_{f^m}(y)\subset \omega$.
	But if $z\in \partial f^k(A)$ then we have two possibilities. If $\dist (\omega,z)>0$ then $z\not \in \overline{\omega}$, In the other case, there exist points $y_n\in \Int f^k(A)$
	such that $\lim_{n\to \infty}\dist (z,\omega_{f^m}(y_n))=0$. Since the space of all $\omega$-limit sets is closed (see\cite{MaiClosed}), we have $z\in \omega_{f^m}(y)=\lim \omega_{f^m}(y_n)\subset f^k(A)$ for some $y$. Since $\omega_{f^m}(y)\subset f^k(A)$, we have two possibilities. Either $\omega_{f^m}(y)\cap \Int f^k(A)\neq \emptyset$, which implies $z\in \omega_{f^m}(y)\subset \omega$ or $\omega_{f^m}(y)\subset \partial f^k(A)$, so we have that $\omega_{f^m}(y)$ is a periodic orbit, so $z$ is periodic point and hence $z\in \omega_{f^m}(z)\subset \omega$ as well. Indeed $\omega$ is a closed set.
	
	Since $m$ fixed the subgraphs with orbits containing basic sets and only intersecting $f^k(A),$ from the last observation we conclude that such intersections are unions of periodic orbits. Hence, the last statement of the Lemma is proved.
	\end{proof}
	
\section{Characterization of $\Rec(\tilde{f})$}\label{sec:characterization}	

In this section, using the results from Section~\ref{sec:recurrent-vs-limitsets}, we finally provide full characterization of set $\Rec(\tilde{f})$.
	 First we deal with nondegenerate recurrent subcontinua whose elements do not all converge to the same circumferential set.
	 Later we deal with the rest of them and close by proving the main theorem of the section, Theorem~\ref{thm:rec-cont}.

\begin{lemma}\label{lem:omegasingleton}
	Let $A$ be a recurrent subcontinuum containing a fixed point $x_0$ such that for each $x\in A,$ sequence $(f^n(x))_{n\geq 1}$ converges to $x_0.$ Then $A$ is periodic point of $\tilde{f}$.
\end{lemma}
\begin{proof}
If $A=\{x_0\}$ then the result is obvious, so suppose that $A$ is not degenerate.
	First we claim that there exist $m\geq 0$ and $\epsilon>0$ such that for each $k\geq m,$ from $\dist_H(f^m(A),f^k(A))<\epsilon$ follows that $f^k(A)$ contains one-sided neighborhoods on the same sides of $x_0$ as $f^m(A).$
If $x_0$ is a branching point then the claim follows from Lemma~\ref{lem:similarimages},  so let us assume that $x_0\in \Int E$ for some edge $E$ of $G$. 	
First suppose $x_0\in \Int f^t(A)$ for some $t$. We have two possibilities. if $E\subset f^k(A)$ for every $k\geq t$ then we 
	take 
	any $\epsilon$ 
	 and put $m=t$. If there is $m\geq t$ such that $E\setminus f^m(A)\neq\emptyset$ and $x_0\in \Int  f^m(A)\neq\emptyset$
		then again we are done by taking sufficiently small $\epsilon$. In the other case we take $m$ such that $x_0\in \partial  f^m(A)\neq\emptyset$ and sufficiently small $\epsilon$
		so that ``hole'' $E\setminus f^m(A)$ is not covered by continuum close to $f^m(A)$.

It is easy to see that by recurrence of $A$, if $f^m(A)$ is periodic then so is $A$, therefore without loss of generality we may assume that $m=0$.

Note that we can always select side $T$ of $x_0$ in such a way that either all of its elements (one-sided neighborhoods) are completely contained in $A$ or do not intersect it. 
For each side $T_i$ of $x_0$ as above, let $\zeta_i$ be the supremum over lengths of one-sided neighborhoods in $T_i$ and let
$\zeta$ be the minimum over all $\zeta_i.$

	Consider a side $T_1$ of $x_0$ such 
	every $T_1$-sided neighborhood $L$ of $x_0$ satisfies $L\subset A$.
	If there exist a $T_1$-sided neighborhood $J_1\subset A$ of $x_0$ and an iterate $m_1$ such that all of its points never leave $A$ under $f^{m_1},$ i.e. $\bigcup_{k\geq 0} f^{km_1}(J_1)\subset A,$ then denote its length by $l_{J_1}.$ If $J_1$ does not exist, put $l_{J_1}=\infty$.
	Apply the above procedure over all sides $T_i$ of $x_0$ for which one-sided neighborhood is a subset of $A$.
	If there are such sides {with $l_{J_k}<\infty$}, let $\kappa_1$ be the minimum of all lengths $l_{J_k}$ and let $m$ be the product over all $m_k,$ otherwise set $\kappa_1=\zeta$ and $m=1$. 
	Note that if $l_{J_k}<\infty$ then for each $J\subset J_k$ and each $r\geq 0$ we have $f^{rm}(J)=f^{r'm_k}(J)\subset A$.
	Denote the $T_i$-sided neighborhood of $x_0$ of length $\kappa_1$ by $K_i$.
	We may renumerate sides of $x_0$ and find $r_1\geq 0$ such that each $1\leq i\leq r_1$ satisfies $f^{ms}(K_i)\subset A$ for all $s\geq 1$.
	Let $\kappa_2<\frac{1}{2}$ be half the distance from $x_0$ to the closest branching point $e$, $x_0\neq e$ and put $\kappa_2=\infty$ if $G$ is interval.
	Denote $\kappa=\min\{\kappa_1,\kappa_2,\zeta,\epsilon\}.$
	Since $A$ is recurrent under $f^m$ there is $n,\ m\vert n,$ such that $\dist_H(f^n(A),A)<\kappa.$	
	
		Since $f^n(x_0)=x_0,$ by continuity of $f^n,$ there exists some $\delta<\kappa$ such that all the points in $\delta$-neighborhood of $x_0$ are mapped by $f^n$ inside a $\kappa$-neighbourhood of $x_0.$ Denote by $L_k$ the $T_k$-sided neighborhood $L_k\subset K_k$ of length $\delta$.
For each $k$ either $f^{ns}(L_k)\subset B(x_0,\kappa)$		
for each $s\geq 1$ or there exists a minimal number $r>0$ such that $f^{(r+1)n}(L_k)\setminus B(x_0,\kappa)\neq \emptyset$.
But in that case there is $j$ such that $f^{rn}(L_k)\supset L_j$. 
By the choice of $\epsilon$
we also have that if $L_k\subset A$ then also $L_j\subset A$.
Assume that $f^{sn}(L_k)$ eventually leaves $A$ for some $s>0$.
	Then without loss of generality (increasing $r$ when necessary) we
may assume that there exists at least one such $j>r_1$ as otherwise 
\[
f^{(r+1)n}(L_k)\subset  (B(x_0,\delta)\cap A)\cup \bigcup_{i=1}^{r_1}f^n(K_i)
\]
and consequently, for each $s>0$ we have 
\[
f^{(r+s)n}(L_k)\subset  (B(x_0,\delta)\cap A)\cup \bigcup_{t=1}^s\bigcup_{i=1}^{r_1}f^{tn}(K_i) \subset A.
\]
In other words, each $L_k$ which eventually leaves $A$ covers at least one $L_j$ with the same property. 
Since the number of sides of $x_0$ is finite, it will imply that there exist $k^\prime$ and $r^\prime\geq 1$ such that $f^{r^\prime n}(L_{k^\prime})\supset L_{k^\prime}.$
Increasing $r^\prime$ when necessary, we obtain a sequence $\{y_i\}_{i\leq 0}\subset A$ such that $f^{r^\prime n}(y_i)=y_{i+1}$ and $\dist(f^{r^\prime n}(y_0),A)>0$.
This would contradict recurrence of $A$ under $f^{r^\prime n}$. We have just proved that if $K_i\cap A\neq \emptyset$ then $i\leq r_1$.

Observe that there is some $n_0$ such that any point from $A$ enters some $\overline{K_k}$ after at most $n_0$ iterates of $f^n.$
Fix any $x\in A$. We can find (sufficiently large) $m$ such that $\dist_H(f^{nm}(A),A)<\kappa,\ \dist(x_0,f^{nm}(x))<\kappa$. 
By continuity $\dist(x_0,f^{nm}(\overline{U_x}))<\kappa$ for some sufficiently small neighborhood $U_x\ni x$.
Then $f^{nm}(\overline{U_x})\subset \cup_{i=1}^{r_1}K_i$. Selecting finite subcover from $U_x$ we obtain $n_0$.
We obtain that $f^ {nn_0}(A)\subset A.$ Due to recurrence, we must have $A=f^ {nn_0}(A)$ which proves that $A$ is periodic.
\end{proof}

	We proceed to show that each recurrent subcontinuum $A,$ which is not a singleton and such that $\omega_f(x)$ is not the same circumferential set for all $x\in A,$ is in fact periodic.

In what follows, we will use a standard method of collapsing of invariant sets, which we briefly present below.
	Let $(G,f)$ be a dynamical system on a graph and $M\subset G$ some non-empty, invariant and closed subset of $G.$
	We can define a factor system $(G/_M,f/_M),$ where $G/_M$ is the factor space and $f/_M\colon G/_M\to G/_M $ is given by $f/_M\circ\pi_M=\pi_M\circ f,$ where $\pi_M\colon G\to G/_M$ is the canonical projection.
	Note that if $M\subset A$ then $A=\pi_M^{-1}(\pi_M(A))$.
	The following lemma follows straightforwardly.
\begin{lemma}\label{lem:projection}
	Let $G$ be a graph, $A\subset G$ its subcontinuum and $M\subset A$ some non-empty, closed and invariant subset such that $M\subset f^n(A),$ for all $n\in\mathbb{N}.$ Then:
\begin{enumerate}[(1)]
\item If $A$ is recurrent in $(G,f),$ then $A/_M$ is recurrent in $(G/_M,f/_M).$
\item If $\omega=\bigcup_{x\in A}(\omega_f(x))$ then $\bigcup_{x\in A/_M}(\omega_{f/_M}(x))=\omega/_M.$
\item If $A$ is recurrent and $f^n(A)/_M$ is periodic in  $(G/_M,f/_M),$ then $A$ is periodic in $(G,f).$
\end{enumerate}
\end{lemma}
\begin{remark}	
	It is not hard to see that if $M$ in Lemma~\ref{lem:projection} has a finite number of components then $G/M$ is topological graph.
\end{remark}

\begin{lemma}
	\label{lem:periodic}
	Let $A\in \Rec(\tilde{f})$ be a nondegenerate continuum. For each $x\in A$ let $M_x\supset \omega_f(x)$ be a maximal $\omega$-limit set. If $M_x$ is not circumferential set for some $x$
	or there are $x,y\in A$ such that $M_x\neq M_y$ then $A$ is a periodic point of $\tilde f$.
\end{lemma}

\begin{proof}
The idea of the proof is to apply  Lemma~\ref{lem:projection} a finite number of times, 
obtaining a sequence of integers $r_i\geq 0$, where $i=0,\ldots,m$ and factor maps $\pi_i \colon (G_i, f_i^{r_i})\to (G_{i+1},f_{i+1})$
and continua $B_i\subset G_i$ such that:
\begin{enumerate}
	\item $B_0=A$, each $B_i$ is recurrent under $f_i$ and $B_m$ is a periodic continuum,
	\item $\pi_i(f_i^{r_i}(B_i))=B_{i+1}$,
	\item 
	$\pi_i^{-1}(B_{i+1})=f_i^{r_i}(B_i)$.
\end{enumerate}

Note that if $B_{i+1}$ is periodic then so is $B_i$, because by definition $f_i^s(f_i^{r_i}(B_i))\subset f_i^{r_i}(B_i)$ 
for some $s>0$ (which is obtained by periodicity of $B_{i+1}$)
which implies that $f^{r_i+s}(B_i)=f^{r_i}(B_i)$ since $f^{r_i}(B_i)$ is recurrent under $f_i^s$. But then $B_i$ is periodic, since it is recurrent as well. 
Indeed, this implies that by the above construction $A$ is periodic, by recursive application of the above argument.
Let us present how the above mentioned maps $\pi_i$ are constructed.	

Let $m,k$ be provided by Lemma~\ref{lem:containscycles} for $f$ and $A$, in particular 
	\[\omega=\bigcup_{x\in f^k(A)}\omega_{f^m}(x)\subset f^k(A)\] 
	and $\omega$ is a closed set.
  Note that $f^k(A)$ is recurrent also for $f^m$, and if $M_x\neq M_y$ then analogous sets calculated for $f^m$ also are not equal. Similarly, if $\omega_f(x)$ is a circumferential set 
   then so is $\omega_{f^m}(x)$.
Therefore, without loss of generality, we may assume that $m=1$. 
By Remark~\ref{rem:evperiodic}, if we show that $f ^k(A)$ is periodic, then so is $A$, therefore without loss of generality we can assume that $k=0.$
Under the above simplifying assumptions we have $\omega=\bigcup_{x\in A}(\omega_f(x))\subset A$, which will simplify the notation.

	Consider some edge $E$  
	and such that $\omega\cap E$ has at least two elements, and fix an ordering of $E$ induced by identification with $[0,1]$. 
	Since $\omega$ is closed, there are $x_1,x_2$, the minimal and maximal points of $E\cap\omega$ in ordering of $E$, respectively.
	Recall that $\omega$ was obtained by an application of Lemma~\ref{lem:containscycles}, hence there are some, possibly degenerate, $\tilde{f}$-periodic subgraphs $C_1\ni x_1$ and $C_2\ni x_2,$ $C_1,C_2\subset A,$ with periods $p_1,p_2\geq 1$ respectively.
	If $C_1\cap C_2\neq\emptyset$ then put $M\coloneqq C_1\cup C_2\subset A$ and $k=p_1p_2$.

	Next suppose that the second case $C_1\cap C_2=\emptyset$ occurs. 
	By Lemma~\ref{lem:connectingcontinua} there is $k\geq 1$ (which is a multiply of  $p=p_1p_2$) and a continuum $J\subset A$ containing $C_1$ and $C_2$ such that
	$J\subset f^k(J)$.	
	Note that if we denote 
	\[
	M\coloneqq \overline{\bigcup_{n\geq 1}f^{kn}(J)}
	\]
	then $M=f^k(M)$
	and $M$ is a continuum.	Note that also $M\subset A$, because continua $f^{kn}(J)$ form a nested sequence and $\liminf_{n\to\infty} d_H(f^{nk}(A),A)=0$.
	
	In both cases, we obtain $k$ and a continuum $M\subset A$ such that  $f^k(M)=M$ and 
	$(C_1\cup C_2)\cap E\subset M\cap E$.
	

Now, let us apply Lemma~\ref{lem:projection} and collapse $M$ to a point, obtaining a factor map $\pi (G,f^k)\to (G/_M,f^k/_M)$ and $L=\pi(E)$ and $\kappa=\pi(\omega)$. There are two possibilities. Either $L\cap \kappa$ is a singleton (in particular, it may happen that $L$ is singleton itself) or $L$ is a loop. Simply, either $M$ covers all of $E\cap \omega$ or it connects $x_1$ with $x_2$
not crossing some point in $E\cap \omega\setminus\{x_1,x_2\}$ producing that way disconnecting interval $E\setminus M$, see Figure~\ref{fig:collapsing}.

In any case, the number of edges in $G/M$ with at least two points intersecting $\kappa$ and which are not loops
	is decreased by at least one compared to $\omega$ and $G$. 
	
	Repeating this procedure a finite number of times, we may assume that $\pi(\omega\cap E)$ is a singleton
for each edge $E$ which is not a loop.
	Therefore, without loss of generality we may assume that we start with a graph with this property, i.e. we replace $G,A,\omega$ by $\pi(G), \pi(A), \pi(\omega)$. 

\begin{figure}[H]
\centering
     \begin{subfigure}[t]{0.45\textwidth}
         \centering
         \includegraphics[scale=0.9]{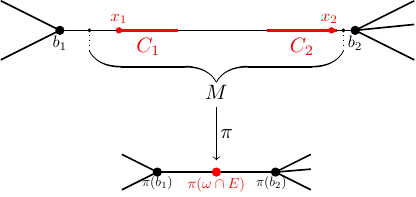}
         \caption{}
         \label{fig:first}
     \end{subfigure}
     \hfill
     \begin{subfigure}[t]{0.45\textwidth}
         \centering
         \includegraphics[scale=0.9]{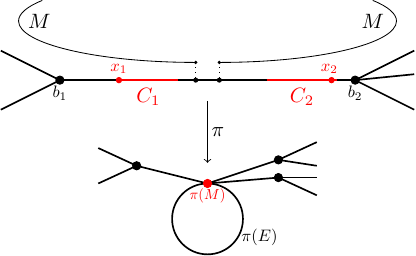}
         \caption{}
         \label{fig:second}
     \end{subfigure}
     
     \caption{ The image of a factor map is shown. $E$ is the edge with endpoints $b_1,b_2.$ All the points of $\omega\cap E$ lie between $x_1$ and $x_2$ in the order of $E$. Figure \ref{fig:first} illustrates the case when $M\subset E$ while Figure \ref{fig:second} illustrates the case when $[x_1,x_2]\setminus M\neq \emptyset$ so $\{b_1,b_2\}\subset M,$ hence $\pi(E)$ is a loop and $\pi(\omega)\cap\pi(E)$ is not necessarily finite.
}\label{fig:collapsing}
\end{figure}
Now, let us deal with loops $L$ having infinite intersection with $\omega$. 
	Note that since $f(\omega)=\omega$, $L\cap f^r(A)$ is always a nondegenerate connected set {for each $r\geq 0$ (namely, }$L$ has only one branching point).
	Suppose first that there is {$r\geq 0$} such that $L\setminus f^r(A)\neq \emptyset$. Then we repeat previous construction for the arc $L\cap f^r(A)$ and the recurrent continuum $f^r(A)$.
	This way we find a connected set $\omega\cap L\subset M\subset f^r(A)$  invariant for $f^k$ for some $k>0$. By the definition $\omega\subset f^r(A)$ so we can collapse $M$ to a point by Lemma~\ref{lem:projection} and continue with continuum $f^r(A)/M$. This way we reduce number of loops intersecting $\omega$ at infinitely many points such that $L\setminus f^r(A)\neq \emptyset$ for some $r$. 
	Furthermore, if we calculate $\omega$ with help of $f^k$
	instead of $f$ resulting set $\omega$ will not increase, so we may assume $k=1$.
	Applying the above reductions a finite number of times, we may assume that if $L$ is a loop and $\omega\cap L$ is infinite, then $f^n(A)\supset L$ for each $n$.

Let us fix a loop $L$ such that $\omega\cap L$ is infinite and $f^n(A)\supset L$ for each $n$.
It follows that for any fixed integer $k$ and for all $n$ we have $f^k(L)\subset f^{n+k}(A)$. Due to recurrence of $A,$ we obtain that $f^k(L)\subset A.$ It follows that $M\coloneqq \overline{\bigcup_{n\geq 1}f^{n}(L)}\subset A$ and $M$ has a finite number of components, since $L\cap f^m(L)\neq \emptyset$ for some $m>0$, by the definition of $\omega$ in its infinite intersection with $L$. Clearly $f(M)\subset M$.
We then collapse $M$ to a point by application of Lemma~\ref{lem:projection}, reducing number of loops intersecting $\omega$
at infinitely many points. In finitely many iterates of this procedure, we are left with $\omega\subset A$ intersecting each edge in a finite number of points.
But $\omega$ is also invariant and closed, so we can collapse it to a point using $M=\omega$ in Lemma~\ref{lem:projection}, obtaining after projection continuum with
a singleton $\omega$, and so this continuum is periodic by Lemma~\ref{lem:omegasingleton}.
This is the last step in the construction of factor maps $\pi_i$, which satisfy all claimed properties. In particular $A$ is periodic, because all factor maps $\pi_i$ were obtained by Lemma~\ref{lem:projection}.
The proof is complete.
\end{proof}

\begin{lemma}\label{lem:recurrent:circ}
 Let $M=E(K,f)$ be circumferential set. Let $C$ be component of $K$  with period $r$ and $A\subset C$ nondegenerate continuum. 
Let $\phi\colon (C,f^r)\to (\mathbb{S}^1,R)$ be an almost conjugacy with an irrational rotation $R$. 
	Then $A$ is recurrent if and only if $A=C$ or \begin{enumerate}
\item\label{rec:first} $A$ has exactly two boundary points $x,y$ in the relative topology of $C$ and $x,y\in M$
\item\label{rec:second} 
$\phi^{-1}(\phi(x))\subset A$ and $\vert\phi^{-1}(\phi(y))\cap A\vert=1$ or vice-versa.
\end{enumerate}
\end{lemma}

	\begin{proof}
Without loss of generality we may assume that $K$ is connected, i.e. $K=C$ and $r=1$. 

By definition $f(C)=C$ so $C$ is recurrent and so the case $A=C$ is trivial.
	Suppose now that $A\neq C$ satisfies the listed conditions and let us prove that it is recurrent.
	Denote $x_n\coloneqq f^n(x),\ y_n\coloneqq f^n(y).$
	Notice that if $\phi^{-1}(\phi(x))$ is a singleton, then by the definition of the almost conjugacy, $x$ is accumulation point of both sets $M\cap A$ and $M\setminus A$.
		Consider now the case when both $\phi^{-1}(\phi(x))$ and $\phi^{-1}(\phi(y))$ are infinite. Then condition $\phi^{-1}(\phi(x))\subset A$ implies that $x$ is an accumulation point of the set $M\setminus A$
		while $|\phi^{-1}(\phi(y)) \cap A|=1$ implies that $M\cap A$ accumulates at $y$, since otherwise $\phi$ cannot be almost conjugacy.
	The set $M$ is minimal and $\lim_{n\to\infty}\diam \phi^{-1}(\phi(x_n))=0$, so there is subsequence $x_{n_k}$ satisfying $\lim_k(x_{n_k})=x$ and we may assume that $x_{n_k}$ converges to $x$ from the mentioned side (i.e. from outside of $A$, contained in  $M\setminus A$). By definition
 $\lim_k(\phi(y_{n_k}))=\phi(y)$ because $R$ is an isometry and hence $y_{n_k}$ converges to $y$ from the corresponding side in $\mathbb{S}^1$.
	Note that the first condition implies that for each $a\in\mathbb{S}^1,$ $\phi^{-1}(a)\cap A\setminus \{x,y\}\neq\emptyset$ implies $\phi^{-1}(a)\subset A$, as otherwise there would be other boundary points of $A$ in $C.$  Since by definition $\phi(x_{n_k})\not\in \phi(A)$ for each $k$, $\phi(y_{n_k})\in \phi(A)$ 
		and therefore $\phi^{-1}(\phi(y_{n_k}))\subset A$.
	But $\lim_{n\to\infty}\diam \phi^{-1}(\phi(y_n))=0$ and therefore $\lim_k(y_{n_k})=y$. This shows that $\lim_{k\to \infty}\diam f^{n_k}(A)\setminus A=0$ and $\lim_{k\to \infty}\diam A\setminus f^{n_k}(A)=0$ which implies
	$\lim(f^{n_k}(A))=A$ in $(C(G),\dist_H)$ (see Figure~\ref{fig:ii}).
	
	
To prove the opposite direction, suppose $A\subset C$ is a recurrent continuum.
	If $A=C$ then we are done so we may assume $A\neq C.$

	First, let us show that $\varphi(A)\neq\mathbb{S}^1.$
	Otherwise, $C\setminus A$ consists of subsets of finitely many fibers of $\varphi.$
	But then $A$ contains almost whole backward orbits of those fibers and hence they are contained in all but finitely many forward images of $A,$ contradicting recurrence.
	 Let $x,y$ be two boundary point of $A$, such that $\phi(x),\phi(y)$ are boundary point of the interval $\phi(A)$.

			Suppose that the first condition does not hold.
			Then there is another boundary point $z$ of $A$ in $C$ besides $x$ and $y.$ 			
			Note that $\phi(A)$ is an interval and so $\phi^{-1}(\mathbb{S}^1\setminus\phi(A))$ is connected as well.
		Note that there is a sequence of points $z_n\in C\setminus A$ such that $\lim_{n\to \infty}z_n=z$.
		It is impossible that $\phi^{-1}(\phi(z))=\{z\}$ because then $\phi(z_n)\in \mathbb{S}^1\setminus \phi(A)$ for all but finitely many $n$ and
		as a consequence $\phi(z)\in \{\phi(x),\phi(y)\}$ which is impossible. 	
			Therefore, the only possibility is that there is $a_0\in \mathbb{S}^1$ such that $\phi^{-1}(a_0)\cap A\setminus\{x,y\}\neq\emptyset$ but $\phi^{-1}(a_0)\not\subset A.$ 	This can happen only for finitely many $a\in \mathbb{S}^1$ and none of them is periodic and $f(\phi^{-1}(a))=\phi^{-1}(R(a))$ for all $a\in \mathbb{S}^1$~\cite[Definition 3.3, Theorem 4.1]{Shao},
	and for all other $b\in \mathbb{S}^1$ we have either $\phi^{-1}(b)\subset A$ or $\phi^{-1}(b)\cap A=\emptyset$.
	This implies that $\phi^{-1}(a_0)\subset f^n(A),$ or $\phi^{-1}(a_0)\cap f^n(A)=\emptyset$ for all but finitely many $n$.
  If $\Int \phi^{-1}(a_0)\cap A\neq \emptyset$, i.e. $\phi^{-1}(a_0)\cap A$ is infinite, it is contradicting the recurrence of $A$.
	The same argument excludes the case $a_0\in \Int \phi(A)$.
So the only possibility is that $\phi^{-1}(a_0)\cap A$ is finite and $\phi^{-1}(a_0)\cap \{x,y\}\neq \emptyset$.
Assume that $x,z\in \phi^{-1}(a_0)$ for some $z\in A$.
Such a situation cannot happen, because if $\Int \phi^{-1}(a_0)\cap A=\emptyset$, then there are points $\lim_n(x_{n})=x$, $x_n\in A$ and with singleton fibers
and so $B(z,\delta)\cap A=\{z\}$ for sufficiently small $\delta>0$ which is impossible.
Indeed (1) holds, and also we easily get that $x,y\in M$ by the above argument.

		From what is already proved, we have that for each $a\in \mathbb{S}^1$ and each $k\geq 0,$ $\phi^{-1}(a)\subset f^k(A)$ or $\vert\phi^{-1}(a)\cap f^k(A)\vert\leq 1.$
	Now suppose that the second assertion does not hold.
	Then exactly one of the following two cases emerges:
\begin{enumerate}[(1)]
	\item\label{caseone} $\phi^{-1}(\phi(x))\not\subset A$ and $\phi^{-1}(\phi(y))\not\subset A$ or
	\item\label{casetwo} both $\phi^{-1}(\phi(x))$ and $\phi^{-1}(\phi(y))$ are infinite subsets of $A.$
\end{enumerate}
	Put $\alpha=\frac{1}{2}\min\{\diam \phi^{-1}(\phi(x)),\diam \phi^{-1}(\phi(y))\}$ and note that $\alpha>0$.
		Since $A$ is recurrent, there is $n\geq 1$ such that $\dist_H(A,f^n(A))<\alpha$ 
	By Remark~\ref{rem:ambient}, if we had \ref{caseone} then also $\phi^{-1}(\phi(x))\not\subset f^n(A)$ and $\phi^{-1}(\phi(y))\not\subset f^n(A)$.
	If we had \ref{casetwo} then both $\phi^{-1}(\phi(x))$ and $\phi^{-1}(\phi(y))$ are subsets of $f^n(A).$
	Since $f^n(A)$ is connected subset of $C$, $d_H(f^n(A),A)<\alpha$ 
	and since obviously $A\neq f^n(A)$ then either $A\subsetneq f^n(A)$ or $f^n(A)\subsetneq A,$ both yielding contradiction with recurrence.
	\begin{figure}[H]
\centering
     \begin{subfigure}[t]{0.45\textwidth}
         \centering
         \includegraphics[scale=1.4]{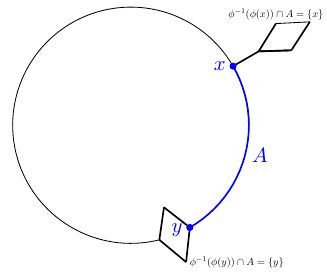}
         \caption{}
         \label{fig:i}
     \end{subfigure}
     \hfill
     \begin{subfigure}[t]{0.45\textwidth}
         \centering
         \includegraphics[scale=1.4]{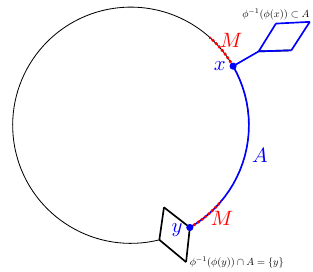}
         \caption{}
         \label{fig:ii}
     \end{subfigure}
     
     \caption{Figure \ref{fig:i} illustrates the case when $A$ fulfills condition \ref{rec:first} but it falls into case \ref{caseone} from the proof and hence is not recurrent. Figure \ref{fig:ii} illustrates the case when $A$ is recurrent.
}\label{fig:recurrent_error}
	\end{figure}
	
\end{proof}

 Let $M=E(K,f)$ be circumferential set. Fix a component $C$ of $K$  with period $r$ and let
 $\phi\colon (C,f^r)\to (\mathbb{S}^1,R)$ be an almost conjugacy with an irrational rotation $R$.
 Denote by $\mathcal{R}_C\subset C(G)$ the set consisting of $C$ and subcontinua $A$ of $C$ 
 which satisfy the following two conditions:
\begin{enumerate}
 	\item $A$ has exactly two boundary points $x,y$ in the relative topology of $C$ and $x,y\in M$
 	\item $\phi^{-1}(\phi(x))\subset A$ and $\vert\phi^{-1}(\phi(y))\cap A\vert=1$ or vice-versa.
\end{enumerate}
Denote by $\mathcal R$ the union of all $\mathcal{R}_C$ over all circumferential sets $M$ and related components $C$.
Under this notation we have the following.
\begin{theorem}\label{thm:rec-cont}
Let $G$ be a topological graph and let $f$ be a graph map on $G$.
	Then 
\begin{equation}
\Rec(\tilde{f})=\Per(\tilde{f})\cup \{\{x\}\colon x\in \Rec(f)\}\cup \mathcal{R}.\label{eq:recAdescription}
\end{equation}
\end{theorem}
	\begin{proof}
For each $x\in G$ let $M_x\supset \omega_f(x)$ be a maximal $\omega$-limit set. 
	Let $A\in\Rec(\tilde{f})$ be a nondegenerate continuum.
	If $M_x$ is not circumferential set for some $x\in A$
	or there are $x,y\in A$ such that $M_x\neq M_y$ then $A\in\Per(\tilde{f})$ by Lemma~\ref{lem:periodic}.
	If there is circumferential set $M$ such that $\omega_f(x)=M$ for all $x\in A$ then $A\in R$ by Lemmas~\ref{lem:contained-cir} and~\ref{lem:recurrent:circ}.
	If $A=\{x\}$ then $x\in\Rec(f).$ This proves inclusion $\subset$ in \eqref{eq:recAdescription}.

Observe that $R\subset \Rec(\tilde{f})$ by Lemma~\ref{lem:recurrent:circ} and now the converse inclusion in \eqref{eq:recAdescription} follows trivially.
\end{proof}

\section{Proof of Theorem~\ref{tm:entropy}}\label{sec:proof}

	In this section we use the results obtained in Section~\ref{sec:characterization}, namely Theorem~\ref{thm:rec-cont}, to prove the main theorem of the paper.
	We begin with a special case of a graph map without periodic points.

\begin{lemma}~\label{lem:circ}
	If $E(K,f)$ is a circumferential set then $(C(K),\tilde f)$ has zero topological entropy.
\end{lemma}
\begin{proof}
	
	Denote by $M=E(K,f)$ the unique minimal set contained in $K$. 
	Assume first that $K$ is connected.
	By\cite{B1} there is a semiconjugacy $\phi\colon (K,f)\to (\mathbb{S}^1,R)$ with the irrational rotation $R$ which is almost conjugacy, that is $\phi^{-1}(x)$
	is connected for every $x$ and $\phi^{-1}(x)\cap M$ has at most $2$ points, and is a singleton for all but countably many points $x$.
	
	Let $A\subset K$ be a continuum. It is clear that if $A$ is recurrent then $\phi(A)$ is recurrent. We know that $(C(\mathbb{S}^1),\tilde R)$ has zero topological entropy (e.g. see \cite[Corollary~30]{KOCSF}).
	Note that since $\phi$ is a monotone map, for every connected set $A\subset K$ we have that $Q=\phi^{-1}(\phi(A))$ is a connected set, and $f(Q)\subset \phi^{-1}(R\circ \phi)(Q)$. 
	It is also clear that if $Q\neq A$ and there is $a\in \mathbb{S}^1$ such that $\phi^{-1}(a)\not\subset A$ but $\phi^{-1}(a)\cap A$ is infinite then $A$ cannot be recurrent (we must have $\phi^{-1}(a)\cap A\subset \partial \phi^{-1}(a)$), {see Lemma~\ref{lem:recurrent:circ}}
	Then any recurrent continuum in $C(\mathbb{S}^1)$ is covered by finitely many recurrent continua in $C(K)$. This means that induced semi-conjugacy $\tilde \phi \colon C(K)\to C(\mathbb{S}^1)$ does not decrease entropy (see \cite{Bow71}), completing the proof.
	
	If $K$ has $s$ connected components, then it is enough to consider $f^s$ which fixes each of the components and then apply previous case. Clearly in this case $C(K)$
	splits into $s$ pairwise disjoint sets, and so the result follows.
\end{proof}

\begin{proof}[Proof of Theorem~\ref{tm:entropy}]
	By the Variational Principle, topological entropy of a dynamical system equals to the supremum of metric entropy over all ergodic measures\cite{Walters}, so in particular:
	\[
	\htop(\tilde{f})=\sup_{A\in \Rec(\tilde{f})}\htop(\tilde{f}_{\omega(A)}).
	\] 
	Since we may view $f$ as a subsystem of $\tilde f$ acting on the space of all degenerate continua, we have $\htop(f)\leq \htop(\tilde{f})$. We have to show that the converse inequality holds.
	Fix any $A\in \Rec(\tilde{f})$. By the above, we may also assume that $A$ is not degenerate.
	We have the following cases:
	\begin{enumerate}
		\item The set $A$ satisfies assumptions of Lemma~\ref{lem:periodic}, then $A$ is a periodic point of $\tilde f$ and as a consequence $\htop(\tilde{f}|_{\omega(A)})=0.$
		\item The set $A$ does not satisfy assumptions of Lemma~\ref{lem:periodic}. Then there is a cycle of graphs $K$ such that $E(K,f)$ is a circumferential set and $\omega_f(x)= E(K,f)$ for each $x\in A$, because $E(K,f)$ is  minimal. Lemma~\ref{lem:contained-cir} yields that $A\subset K$.
		Then the result follows by Lemma~\ref{lem:circ}.
	\end{enumerate}
	The proof is complete.
\end{proof}

\section*{Acknowledgements}
Research of P. Oprocha was supported by National Science Centre, Poland (NCN), grant no. 2019/35/B/ST1/02239.

\bibliographystyle{amsplain}
\bibliography{EntropyContinua_06}
\end{document}